\documentclass{amsart}

\usepackage{amsmath,amssymb,graphicx,latexsym,mathrsfs}
\usepackage[all]{xy}
\usepackage[bbgreekl]{mathbbol}

\newtheorem{thm}{Theorem}[section]
\newtheorem{cor}[thm]{Corollary}
\newtheorem{lem}[thm]{Lemma}
\newtheorem{prop}[thm]{Proposition}
\theoremstyle{definition}
\newtheorem{defn}[thm]{Definition}
\newtheorem{example}[thm]{Example}
\theoremstyle{remark}
\newtheorem{rem}[thm]{Remark}

\numberwithin{equation}{section}

\newcommand{\abs}[1]{\left\vert#1\right\vert}

\newcommand{\eps}{\varepsilon}
\newcommand{\To}{\longrightarrow}
\newcommand{\tof}[1]{\stackrel{#1}{\longrightarrow}}

\newcommand{\A}{\mathcal{A}}

\newcommand{\Hb}{{\mbox{$H^{uf}\!\!$}}}
\newcommand{\Cb}{{\mbox{$C^{uf}\!\!$}}}

\def\Cb{{\mathbb C}}

\def\Hb{{\mathbb H}}
\def\Nb{{\mathbb N}}
\def\Rb{{\mathbb R}}

\def\Zb{{\mathbb Z}}

\def\Zb{{\mathbb Z}}

\DeclareMathOperator{\Hom}{Hom}
\newcommand{\mto}[1]{\stackrel{#1}{\rightarrow}}

\def\A{{\mathcal {A}}}
\def\Ac{{\mathcal A}}
\def\Bc{{\mathcal B}}
\def\Cc{{\mathcal C}}
\def\Dc{{\mathcal D}}
\def\Ec{{\mathcal E}}
\def\Fc{{\mathcal F}}
\def\Gc{{\mathcal G}}
\def\Hc{{\mathcal H}}

\def\Lc{{\mathcal L}}

\def\Pc{{\mathcal P}}

\def\Sc{{\mathcal S}}

\def\As{{\mathscr A}}
\def\Bs{{\mathscr B}}

\def\dbar{{\overline{\partial}}}

\DeclareMathOperator{\Ho}{Ho}

\newcommand{\kb}{\mathbb{k}}
\newcommand{\Rs}{\mathscr{R}}
\newcommand{\Gs}{\mathscr{G}}

\newcommand{\lra}{\longrightarrow}
\newcommand{\arrows}{\rightrightarrows} 

\newcommand{\Adc}{(\Adot,d_\Ac,c_\Ac)}
\newcommand{\Bdc}{(\Bdot,d_\Bc,c_\Bc)}
\newcommand{\psas}{\mathcal{P}\negthickspace_{\mathscr{A}}}
\newcommand{\psa}{\mathcal{P}\negthickspace_{\mathscr{A}}}
\newcommand{\hPb}{h\mathcal{P}\negthickspace_{\mathscr{B}}}
\newcommand{\hPa}{h\mathcal{P}\negthickspace_{\mathscr{A}}}

\newcommand{\Deli}{\mathcal{D}} 
\newcommand{\ddeli}{d^{\text{\tiny{Del}}}} 
\newcommand{\dcechdeli}{D}  

\newcommand{\Adot}{{\mathcal{A}^\bullet}}
\newcommand{\Bdot}{{\mathcal{B}^\bullet}}
\newcommand{\Edot}{E^\bullet}

\newcommand{\Xdot}{X^\bullet}

\newcommand{\dcech}{\check{\delta}} 
\newcommand{\defined}{:=}

\newcommand{\Gcech}{\check{\Gc}}  
\newcommand{\dr}{d^{\text{\tiny{DR}}}}   
\newcommand{\DP}{D_P}  
\newcommand{\incl}{\hookrightarrow}
\newcommand{\isom}{\simeq}
\newcommand{\Uone}{{U(1)}}

\newcommand{\tensA}{\negmedspace\otimes_{\negmedspace A}\negmedspace}
\newcommand{\tensBc}{\negmedspace\otimes_{\negmedspace \mathcal{B}}\negmedspace}
\newcommand{\tensAc}{\negmedspace\otimes_{\negmedspace \mathcal{A}}\negmedspace}
\newcommand{\tensB}{\negmedspace\otimes_B\negmedspace}

\newcommand{\tens}{\negmedspace\otimes\negmedspace}
\newcommand{\LA}{\mathcal{L}\!_A\!}

\newcommand{\LAdot}{\mathcal{L}\!_{\Ac^\bullet}\!}

\newcommand{\LB}{\mathcal{L}\!_B\!}
 \DeclareMathOperator{\bor}{Bor}
\DeclareMathOperator{\Mod}{Mod} \DeclareMathOperator{\Free}{Free}
\DeclareMathOperator{\forget}{Forget}
 
\DeclareMathOperator{\Ch}{Ch}

\DeclareMathOperator{\opp}{op}
\DeclareMathOperator{\Dol}{\mathcal{D}}
\DeclareMathOperator{\ddol}{d^{\text{\tiny{Dol}}}}
\DeclareMathOperator{\Coh}{Coh}
\DeclareMathOperator{\dlog}{dlog}
\DeclareMathOperator{\hFin}{h-Fin}
\DeclareMathOperator{\tot}{tot} 
\DeclareMathOperator{\rad}{rad} 

\begin{document}

\title{Mukai duality for gerbes with connection}

\author{Jonathan Block and Calder Daenzer}
\thanks{C.D. partially supported by NSF grant DMS-0703718.}
\address{J.B.: Department of Mathematics, University of Pennsylvania, Philadelphia, PA 19104\\
C.D.: Department of Mathematics, 970 Evans Hall, Berkeley, CA 94720-3840}%
\email{blockj@math.upenn.edu;  cdaenzer@math.berkeley.edu}

\maketitle

\begin{abstract}
We study gerbes with connection over an \'etale stack via noncommutative algebras of differential forms on a groupoid presenting the stack.
We then describe a dg-category of modules over any such algebra, which we claim represents a dg-enhancement of the derived category of coherent analytic sheaves on the gerbe in question.

This category can be used to phrase and prove Fourier-Mukai type dualities between gerbes and other noncommutative spaces.
As an application of the theory, we show that a gerbe with flat connection on a torus is dual (in a sense analogous to Fourier-Mukai duality or T-duality) to a noncommutative holomorphic dual torus.
\end{abstract}
\tableofcontents
\section{Introduction}

Two varieties $X$ and $Y$ are Mukai dual when there is a sheaf $\Pc$
on $X\times Y$ which induces an equivalence between bounded derived categories of coherent sheaves:
\[ \phi_\Pc:D^b_{\Coh}(X)\To D^b_{\Coh}(Y) \]
The functor $\phi_\Pc$ takes a sheaf on $X$, pulls it back to a sheaf on $X\times Y$, tensors with $\Pc$, then pushes forward to $Y$, all in a derived way.  The basic example of such an equivalence is that between a complex torus $T=V/\Lambda$ and its dual $T^\vee:=\Hom(\Lambda,U(1))$.

In several situations a space that one might expect to have a Mukai dual does not have one.  For example a principal torus bundle ought to have a dual obtained by dualizing the fibers of trivialized coordinate patches.  This procedure is not compatible with the global structure of the bundle, however, and the attempt to glue together a global dual object results not in a space but in a $U(1)$-gerbe over a family of dual tori (see e.g. \cite{DP},\cite{Pol},\cite{Dae}).  Furthermore, once gerbes are brought into the picture one is tempted to ask: what are the Mukai duals of a gerbe?  In general the dual to a gerbe, when it can be made sense of at all, is a noncommutative space (in the sense of noncommutative geometry).

This motivates the following questions, whose answer is the purpose of this paper:
\begin{enumerate}
\item What is the analogue of the bounded derived category for gerbes and noncommutative spaces?
\item Assuming the first question is answered, can an analogue of Mukai duality be defined for gerbes and noncommutative spaces?
\end{enumerate}

The general framework for answering these questions is the following.  First we present gerbes and noncommutative spaces with complex structure by a triple of data
\[ \As=(\Adot,d,c)=(\textrm{associative graded algebra, derivation, curving})\]
called a curved differential graded algebra (curved dga).  We then form a dg-category $\psa$ of modules over any such curved dga.  The objects of $\psa$ are modules satisfying finite projective type conditions which are a noncommutative version of being vector bundles, and they are equipped with an operator that encodes the holomorphic data.  This dg-category answers the first question in the strong sense that its homotopy category $\Ho(\psa)$ is a generalization to gerbes and noncommutative spaces of the bounded derived category of coherent sheaves.   Thus $\psa$ is a dg-enhancement of the derived category, and is a useful theoretical tool even in the case of complex manifolds.  As for the second question, Mukai duality between gerbes and noncommutative spaces is expressed as a quasi-equivalence between the associated dg-categories.

Here is the theorem from the commutative world which backs the claim that $\psa$ is a dg-enhancement of the derived category of coherent sheaves:

Let $X$ be a compact complex manifold and let $T_X^{0,1}$ denote the antiholomorphic cotangent bundle of $X$.  Then the complex structure on $X$ is encoded in the Dolbeault algebra $\Adot=\Gamma^\infty(X;\wedge^\bullet(T^{0,1}_X))$ together with its $\dbar$-operator.  The triple $\As(X):=(\Adot,\dbar,0)$ forms a unital curved dga, and in this case the associated dg-category $\Pc_{\negthickspace\As(X)}$ satisfies:
\begin{thm}
\cite{Bl1} The homotopy category of $\Pc_{\negthickspace\As(X)}$ is equivalent to the bounded derived category of coherent analytic sheaves on $X$.
\end{thm}
Thus in this case objects of $\psa$ correspond to quasi-coherent complexes of analytic sheaves, and the morphisms between any two such objects form a complex that computes Ext groups.

The strategy of encoding geometric structures in a curved dga and phrasing duality in terms of quasi-equivalences between associated dg-categories is developed in \cite{Bl1}.  The curved dga of a holomorphic noncommutative torus is described in \cite{Bl2}.  Thus the general framework is understood, and we have commutative examples (the Dolbeault algebra of a complex manifold) as well as noncommutative examples (the holomorphic noncommutative torus).

The remaining work which we complete here is first to define the curved dga associated to a gerbe, and second, after noting that the resulting curved dga is nonunital,  to make the necessary modifications to $\psa$ for the nonunital case.  The construction of the curved dga involves a choice of a connection on the gerbe, but we show that the curved dga is independent (in an appropriate sense) of this choice.

As an application of the theory we will be able to complete the proof begun in \cite{Bl2} that a holomorphic noncommutative torus is Mukai dual to a gerbe on a dual torus.  This fact is not unexpected; its analogue in formal complex geometry has already been proved \cite{BBP}.

Before proceeding with general constructions, let us give the motivating example for our construction of the curved dga associated to a gerbe on a manifold.  This example will also show how connections on gerbes come into play.

According to Giraud \cite{Gir}, the $U(1)$-gerbes on a paracompact space $X$ are classified by degree two \v{C}ech cohomology with values in the sheaf of $U(1)$-valued functions, $H^2(X;U(1))$.  According to Dixmier and Douady \cite{DD}, $H^2(X;U(1))$ parameterizes the Morita equivalence classes of continuous trace $C^*$-algebras with spectrum $X$.  A groupoid called the \v{C}ech groupoid mediates between these two classifications.  Let us see how this works.  Let $\{U_i\}$ be a cover of $X$.  The \v{C}ech groupoid of the cover $\{U_i\}$ is defined as follows:
\[
\Gcech=(\Gcech_1\arrows\Gcech_0)\defined(\coprod_{\langle i,j\rangle}U_{ij}\arrows\coprod_iU_i)  \]
where the source and target maps are the inclusions
\[\begin{cases}
&s:U_{ij}\incl U_j\\
&t:U_{ij}\incl U_i.
\end{cases}\]
Note that the space $\check{\Gc}_2$ of composable arrows is just $\coprod U_{ijk}$.  A degree two cohomology class on $X$ can be presented as a \v{C}ech 2-cocycle $\sigma:\coprod U_{ijk}\to U(1)$ for some cover $\{U_i\}$ of $X$, which is the same as a groupoid 2-cocycle $\sigma:\check{\Gc}_2\to U(1)$ (see Section \ref{S:twistedDR}).  This data determines a groupoid extension $U(1)\rtimes^\sigma\check{\Gc}$ analogous to the group extension determined by a group 2-cocycle, and this groupoid is a presentation of the gerbe $\Gs(\sigma)$ associated to $[\sigma]\in H^2(X;U(1))$ in Giraud's program.  We mean presentation in the sense that $\Gs(\sigma)$ is equivalent to the stack of $(U(1)\rtimes^\sigma\check{\Gc})$-torsors.

On the other hand, let $C_c(\check{\Gc})$ denote the compactly supported complex valued functions on the space $\check{\Gc}_1=\coprod U_{ij}$.  Then $\sigma$ determines an associative but noncommutative multiplication.
\begin{equation}\label{E:IntroMult}
a*b(x,i,k):=\sum_j a(x,i,j)b(x,j,k)\sigma(x,i,j,k) \end{equation} where $a,b\in C_c(\Gcech)$ and $(x,i_1,\dots,i_n)$ denotes $x\in U_{i_1,\dots,i_n}\subset\Gcech_n$.
The resulting algebra, called a twisted groupoid algebra, has an involution and a norm for which it can be completed to a $C^*$-algebra \cite{Ren}.  This $C^*$-algebra is a continuous trace algebras with spectrum $X$.  The cohomology class $[\sigma]\in H^2(X;\Uone)$ is recovered as the Dixmier-Douady invariant of this $C^*$-algebra, and is a complete invariant of its Morita equivalence class.  Thus this algebra can be treated as a presentation of a gerbe.

In this paper we are describing gerbe theoretic phenomena, but it will be necessary to work with algebras.  So we will use twisted groupoid algebras and think of them as presentations of gerbes.  This means, in particular, that stacks will not appear explicitly.  Instead, constructions will be at the presentation level and then shown to be invariant under Morita equivalence.

Now suppose $X$ is not just a paracompact space, but is a smooth complex manifold, and suppose further that the 2-cocycle $\sigma$ is a smooth function.  Then to describe complex structure on the associated gerbe over $X$ we combine the ideas of the twisted groupoid algebra (which presents a $U(1)$-gerbe) and the Dolbeault algebra with $\dbar$-operator (which encodes complex structure).  We call the resulting graded algebra a twisted Dolbeault algebra on $X$.
\[ \Adot:=\Gamma_c^\infty(\Gcech,\sigma; t^*(\wedge^\bullet T_{\Gcech_0}^{0,1})). \]
This is the vector space of smooth compactly supported sections over $\check{\Gc}_1$ of the pullback via the target map of the bundle $\wedge^\bullet T_{\Gcech_0}^{0,1}$.  The multiplication is twisted as in Equation \eqref{E:IntroMult} except with wedge product in the fibers.

There is a problem, however.  The $\dbar$-operator from the complex structure on the manifold $\check{\Gc_1}$ is not derivational with respect to the twisted multiplication.  But after a possible refinement of covers it is possible to find a $(0,1)$-form $\theta$ on the manifold $\check{\Gc}_1=\coprod U_{ij}$ such that $da(x,i,j):=\dbar a(x,i,j)+\theta(x,i,j)\wedge a(x,i,j)$ \emph{is} a derivation for the twisted multiplication.
In general $d$ has a curving.  That is, there is a $(0,2)$-form $B$ on $\check{\Gc}_1$ called the curving of the connection, such that $d^2 a=[B,a]$.  The choice of $\theta$ and $B$ which restore the derivation and describes its curving is interpreted as a choice of connection on the gerbe.

This data $(\Adot,d,B)$ of a twisted Dolbeault algebra $\Adot$, a derivation $d=\dbar+\theta$, and its curving $B$, is the curved dga associated to a gerbe with $\dbar$-connection on a complex manifold $X$.  It is this type of object (and generalizations thereof) whose derived category we will describe, and whose Mukai partners we would like to discover.

Here is an outline of the paper:

In Section \ref{S:twistedDR} we formalize the notion of the twisted Dolbeault algebra associated to a gerbe with $\dbar$-connection on a complex manifold.  As was remarked above, this is based on an algebra of functions on the \v{C}ech groupoid of a cover of the manifold.   The construction works exactly the same for arbitrary complex \'etale groupoids, so we develop the theory at that level of generality.  In particular the construction is valid for gerbes with connection on orbifolds.  We also present a smooth version of this, which we call the twisted de Rham algebra.

Next we want to describe the derived category of coherent sheaves on a gerbe with $\dbar$-connection, or rather its dg-enhancement $\psa$.  There is an unavoidable problem, however:  the dg-category of modules $\psa$ over a curved dga was defined in \cite{Bl1} only for unital algebras, whereas groupoid algebras are nonunital (and adjoining a unit is a bad idea for several reasons).  The difficulties this presents are in fact rather severe.  They are resolved in Sections \ref{S:bornoalg}-\ref{S:bornoprops}, which comprise the technical heart of the paper.  The first three of these sections describe the appropriate fixes for a general nonunital curved dga and its associated modules.  Section \ref{S:bornoprops} shows that the twisted Dolbeault dga's (and some other groupoid algebras) fit within the class of nonunital algebras for which the fix works.

In Section \ref{S:threelevels} we address a hierarchy of equivalences between curved dga's as they apply to twisted Dolbeault and de Rham algebras.  The strongest equivalence is an isomorphism corresponding to gauge transformation of the connection.  The second level equivalence is a Morita equivalence that is induced by a Morita equivalence between groupoids presenting a gerbe with connection; this one implies that the twisted Dolbeault algebra over a groupoid $\Gc$ is a presentation of a (stack theoretic) gerbe with connective structure as in \cite{Br}.  The third and weakest form of equivalence between two curved dga's is a quasi-equivalence between the associated dg-categories.  This is taken as the definition of duality for curved dga's, and when applied to a twisted Dolbeault dga it should be interpreted as the noncommutative analogue of Mukai duality.

In Section \ref{S:torus} we apply the theory to prove that a gerbe with flat $\dbar$-connection on a torus is dual to a holomorphic noncommutative dual torus.  More precisely, we show that given any gerbe with flat $\dbar$-connection on a torus there is an associated holomorphic noncommutative dual torus, and that there are curved dga's corresponding to both objects, and the associated dg-categories are quasi-equivalent.\newline

\noindent\textbf{Acknowledgments.}\ C. Daenzer would like to thank Oren Ben-Bassat, Peter Dalakov, Ralf Meyer, Marc Rieffel, and Jim Stasheff for help and discussion relating to this project.  We would also like to thank the referee for numerous insightful comments.


\section{The curved dga corresponding to a gerbe with connection }\label{S:twistedDR}
The twisted Dolbeault algebra approach for presenting a gerbe with connection that was outlined in the introduction can be generalized to present a gerbe with connection on a complex \'etale groupoid.  The input data is a complex \'etale groupoid and a certain 2-cocycle in groupoid hypercohomology.  The output is a twisted version of the Dolbeault algebra with modified $\dbar$-operator.  There is also a smooth version of this, which gives a twisted exterior algebra of differential forms and a modified de Rham operator.  In this section we describe these algebras in detail and show that they can be endowed with the structure of a curved dga.

\subsection{\v{C}ech cohomology of an \'etale groupoid}

A \textbf{smooth \'etale groupoid} is a (small) groupoid $\Gc=(\Gc_1\arrows\Gc_0)$ whose sets of objects and arrows are both smooth manifolds, and all of whose structure maps are local diffeomorphisms.  All \'etale groupoids considered here will be smooth.  The $n$-tuples of composable pairs in $\Gc$ will be denoted $\Gc_n$, though by abuse of notation $\Gc$ will mean both the groupoid and the manifold $\Gc_1$ of arrows.  The source and target maps are denoted by $s$ and $t$, and their extensions to $\Gc_n$ are
\[ s_n,t_n:\Gc_n\To\Gc_0;\quad s_n(g_1,\dots,g_n):=s(g_n);\quad t_n(g_1,\dots,g_n):=t(g_1). \]

Let $(\Gc\arrows\Gc_0)$ be an \'etale groupoid.  A \textbf{(left) $\Gc$-sheaf} is a sheaf $F$ of abelian groups on $\Gc_0$ (viewed as an \'etale space $F\stackrel{\eps}{\to}\Gc_0$) together with a map $\Gc\times_{s,\eps}F\to F$ satisfying conditions analogous to a group action, that is, such that the two obvious ways to get from $\Gc_2\times_{s,\eps}F$ to $F$ agree.  The sequence of abelian groups \\
$C^n(\Gc;F):=
\begin{cases}
\{f\in\Gamma(\Gc_n;t_n^*F)\:|\:f(g_1,\dots,g_n)=0
\text{ if } g_i\in\Gc_0 \text{ for any i }\}&n>0 \\
\Gamma(\Gc_0;F) &n=0, \\
\end{cases} $

\noindent where $\Gamma$ denotes smooth sections, forms a complex with the differential
\begin{align*}
\dcech f(g_1,\dots,g_{n+1})&:= g_1\cdot f(g_2,\dots,g_{n+1}) \\
& +\sum_{i=1\dots n}(-1)^if(g_1,\dots,g_ig_{i+1},\dots,g_{n+1}) +(-1)^{n+1}f(g_1,\dots,g_n)
\end{align*}
for $f\in C^n(\Gc;F),\:n>0$, and
\[ \dcech f(g):=g\cdot f(sg)-f(tg) \]
when $f\in C^0(\Gc;F).$

The cohomology $\check{H}(\Gc;F)$ of this complex will be called the \textbf{\v{C}ech-groupoid cohomology} of $\Gc$ with coefficients in $F$.  This is a generalization of \v{C}ech cohomology.  Indeed, if $\Gc$ is the \v{C}ech groupoid of a cover of some manifold, the \v{C}ech-groupoid complex agrees with the classical \v{C}ech complex and computes classical \v{C}ech cohomology of $F$ with respect to the cover.  Just as the cohomology groups of a cover are not equal to the sheaf cohomology of $F$ unless the cover is ``good'' (or more generally F-acyclic), \v{C}ech-groupoid cohomology is not in general equal to the groupoid (or stack) cohomology described in \cite{MC} or \cite{Beh} unless the groupoid is ``good'' in an appropriate sense.  One manifestation of this failure is that \v{C}ech-groupoid cohomology is not invariant under Morita equivalence of groupoids (which is a generalization of refinement of a cover).  Conditions under which a groupoid should be deemed good are addressed in Section \ref{S:stacks}, but it is useful to allow all groupoids at this time.

\subsection{\v{C}ech-Deligne cohomology}\label{S:CechDelingeCoho}

To describe our presentation of a gerbe with connection we use the \v{C}ech-groupoid (hyper)cohomology of the Deligne complex.
\begin{defn}\label{D:DeligneComplex}
The \textbf{$n^{th}-$Deligne complex} $(\Deli_{(n)},\ddeli)$ is the complex of sheaves of abelian groups
\[
 1\to\Uone\stackrel{\dlog}{\lra}2\pi\sqrt{-1}\Ac^1\stackrel{\dr}{\lra}
          \dots\stackrel{\dr}{\lra}2\pi\sqrt{-1}\Ac^n\lra 0
\]
where $\Ac^q$ is the sheaf of differential $q$-forms, $\dr$ is the de Rham differential, and $\dlog(f):=f^{-1}\dr f$.
In what follows, we are only using $\Deli_{(2)}$, so we write $\Deli\defined\Deli_{(2)}$.
\end{defn}

If $M\stackrel{\phi}{\to} N$ is a map of smooth manifolds, then pullback of differential forms gives a map $\Deli(\phi):\Deli_N\to\Deli_M$, thus for an \'etale groupoid $\Gc\arrows \Gc_0$, $\Deli_{\Gc_0}$ is naturally a $\Gc$-sheaf.  Indeed, let $x\stackrel{\gamma}{\to}y$ be an arrow in $\Gc$.  Using the local diffeomorphism property of \'etale groupoids, one sees that $\gamma$ induces a unique diffeomorphism $(U_x,x)\stackrel{\tilde{\gamma}}{\to}(U_y,y)$ between small enough pointed neighborhoods of $x$ and $y$.  This, in turn, induces the $\Gc$-action
\[
\Gc\times_s\Deli_{\Gc_0}\lra \Deli_{\Gc_0}
\]
\[ \gamma\cdot f:=\Deli(\tilde{\gamma}^{-1})(f)\in \Deli_{{\Gc_0},t\gamma}\quad\text{ for }f\in\Deli_{{\Gc_0},s\gamma}.
\]
Furthermore, the \v{C}ech and Deligne differentials commute, so the groups
\[ C^{p,q}(\Gc;\Deli):=C^p(\Gc;\Deli^q)\]
form a double complex with differential $\dcechdeli^{p,q}:=\ddeli+(-1)^q\dcech$.  The cohomology $\check{\Hb}(\Gc;\Deli):=H(\tot(C^{\bullet\bullet}(\Gc;\dcechdeli))$ will be called the \textbf{\v{C}ech-Deligne cohomology} of $\Gc$.

Of primary interest are the \v{C}ech-Deligne 2-cocycles.  Such a 2-cocycle on a groupoid $\Gc$ is the data for what we call a \textbf{gerbe with connection} on $\Gc$
\footnote{It would be more accurate to call this the data for a presentation of a gerbe with connection on the stack $B\Gc$ associated to $\Gc$, but we prefer to avoid that cumbersome terminology for the moment.  To avoid confusion, we will refer to the gerbes defined in \cite{Gir} (of which our gerbes are presentations) as \textbf{stack theoretic gerbes}.}.
It is given by a triple
\begin{equation}
\label{E:2cocycle}(\sigma,\theta,c)\in\Uone(\Gc_2)\times2\pi\sqrt{-1}\Ac^1(\Gc_1)\times2\pi\sqrt{-1}\Ac^2(\Gc_0) \end{equation}
satisfying
\begin{equation}
\dcech\sigma=1;\quad \dlog\sigma=\dcech\theta;\quad \dr\theta=-\dcech c.
\end{equation}
Note that $c$ is not necessarily $\dr-$closed.

\subsection{The twisted de Rham algebra}

Now we will describe the twisted de Rham algebra associated to a \v{C}ech-Deligne 2-cocycle on an \'etale groupoid $\Gc$.  This algebra will naturally be equipped with the structure of what is called a curved differential graded algebra (dga).

\begin{defn}\label{D:cdga}
A \textbf{curved dga} is a triple $\As=(\A^\bullet,d,c)$, where $\A^\bullet$ is an $\Nb-$graded, unital, associative algebra over a field $k$, equipped with a degree one map
\[d:\A^\bullet\to \A^{\bullet+1}\]
satisfying
\begin{itemize}
\item[(i)] $d(ab) = d(a)b + (-1)^{|a|}ad(b)$
\item[(ii)]$d^2(a)=[c,a];\quad c\in\A^2$
\item[(iii)] $dc=0$.
\end{itemize}
Condition (i) is the Leibnitz rule.  We call the fixed element $c\in\A^2$ the \textbf{curving} of the dga, and then condition (iii) is the requirement that $c$ satisfy the Bianchi identify. We write $\A$ for the degree zero component, $\A^0$, of the curved dga.  If (iii) is not satisfied, we will call the dga a \textbf{really curved dga}.
\end{defn}
For the case when $\Ac$ is nonunital, there are two changes to be made.  The first is that the curving no longer lies in $\A^2$ but is instead in a multiplier algebra of $\A^2$.  The second is the introduction of some analysis, which is necessary to get meaningful Morita equivalences between nonunital curved dga's.  These details will be carefully addressed, but let us first get to the example of the twisted de Rham dga.

\begin{defn}\label{D:twisteddeRhamAlg} Let $(\sigma,\theta,B)$ be a \v{C}ech-Deligne 2-cocycle on an \'etale groupoid $(\Gc\arrows \Gc_0)$ as in \ref{E:2cocycle}.  The \textbf{twisted de Rham dga} associated to $(\sigma,\theta,B)$ is the following (nonunital) really curved dga:
\begin{align}\label{E:twisteddeRhamAlg}
\As(\Gc,(\sigma,\theta,B))=(\Adot,d,c):=(\Gamma_c^\infty(\Gc,\sigma;t^*(\wedge^\bullet T^*_{\Gc_0,\Cb})),
\dr+\theta,B)
\end{align}
The first piece, $\Adot$, is called the associated \textbf{twisted de Rham algebra}.  As a vector space, $\Ac^\bullet$ is the smooth compactly supported sections of the pullback via $t$ of the complexified de Rham algebra over $\Gc_0$.  $\Adot$ has multiplication given by $\sigma$-twisted convolution on the groupoid and wedge product in the coefficients, and it is equipped with a derivation that is the de Rham differential modified by $\theta$, and this derivation has curving $B$.  More explicitly, we have, for $f,g\in\Adot$ and $\gamma,\eta,\tau\in\Gc$:
\begin{enumerate}
\item Multiplication:\ \ $f*g(\gamma):=\sum_{\eta\tau=\gamma}\sigma(\eta,\tau)f(\eta)\wedge \eta\cdot g(\tau)$
\item Derivation:\ \ $df(\gamma):=(\dr f)(\gamma)+\theta(\gamma)\wedge f(\gamma)$
\item Curving:\ \ $d^2f=d(\dr f+\theta\wedge f)$
\begin{align*}            &=(\dr)^2f+\dr\theta\wedge f
   -\theta\wedge\dr f+\theta\wedge\theta\wedge f+\theta\wedge\dr f \\
                          &=\dr\theta\wedge f.
\end{align*}
\end{enumerate}
\end{defn}
\noindent In fact $\Adot$ is a $*$-algebra, with the \newline
\begin{enumerate}
\item[(4)] Involution:\ \ $f^*(\gamma):=\sigma(\gamma^{-1},\gamma)^{-1}\overline{\gamma\cdot f(\gamma^{-1})}.$
\end{enumerate}
\begin{prop}\label{P:twistedDRalg} The twisted de Rham dga is a $($nonunital\thinspace$)$ really curved dga.  It is a curved $($but not really curved\thinspace$)$ dga precisely when $\dr B=0$.
\end{prop}
\begin{proof}
One checks the following facts directly.
\begin{enumerate}
\item The multiplication is associative if and only if $\dcech\sigma=1,$ so $\Adot$ is indeed an associative algebra.
\item $d$ is a derivation (i.e. satisfies Leibnitz) for this multiplication if and only if $\dlog\sigma=\dcech\theta,$
\item The curving, $d^2f=\dr\theta\wedge f$ is equal to $-\dcech B\wedge f$ exactly when $\dr\theta=-\dcech B$, and by definition of multiplication, $-\dcech B\wedge f =B*f-f*B=:[B,f].$
\end{enumerate}
So a \v{C}ech-Deligne 2-cocycle was exactly what we needed.
\end{proof}
Note that this algebra is nonunital since it is based on functions with compact support.  Note also that the curving $B$ is a smooth 2-form on $\Gc$ (whose support is in $\Gc_0\subset\Gc$) which is not compactly supported, thus $B$ lies in the multiplier algebra of $\Ac^\bullet$.

\subsection{The twisted Dolbeault algebra}
If $\Gc$ is a complex \'etale groupoid, meaning that $\Gc$ is complex manifold and all structure maps are locally biholomorphisms, then instead of using the smooth Deligne complex we can use the \textbf{truncated Dolbeault complex}
\[
(\Dol,\ddol):=\quad 1\to\Cb^*\stackrel{\dbar\log}{\lra}\Ac^{0,1}\stackrel{\dbar}{\lra}\Ac^{0,2}\to 0.
\]
Here $\Ac^{p,q}$ is the sheaf of differential $(p,q)$-forms, $\dbar$ is the  Dolbeault differential, and $\dbar\log f:=f^{-1}\dbar f$.  A 2-cocycle in the groupoid cohomology of this complex,
\begin{equation}\label{E:Dolbeault2cocycle}
 (\sigma,\theta,B)\in\Cb^*(\Gc_2)\times \Ac^{0,1}(\Gc_1)\times\Ac^{0,2}(\Gc_0),
\end{equation}
is the data for what we call a \textbf{gerbe with $\dbar$-connection} on $\Gc$.
\begin{defn}\label{D:dolbeaultAlg} Just as in the smooth case, one forms a really curved dga, which we call the \textbf{twisted Dolbeault dga} associated to $(\sigma,\theta,B).$
\begin{align}\label{E:dolbeaultAlg}
\As(\Gc,(\sigma,\theta,B))=(\Ac^\bullet,d,c):=(\Gamma_c^\infty(\Gc,\sigma;t^*(\wedge^\bullet T^{0,1}_{\Gc_0})),
\dbar+\theta,B)
\end{align}
\end{defn}
$\Adot$ is called the \textbf{twisted Dolbeault algebra}. It is the space of smooth compactly supported sections of the  Dolbeault algebra, with twisted multiplication as in the de Rham algebra situation.   The derivation $d$ is given by $df(\gamma):=\dbar f(\gamma)+\theta(\gamma)\wedge f(\gamma)$.  The twisted Dolbeault algebra also has an involution, defined just as in the de Rham case.

The interesting thing here is that the curving $B$ need not be $\dr-$closed for the twisted Dolbeault dga to be a curved (but not really curved) dga.  Instead $B$ only needs to be $\dbar$-closed, that is holomorphic.

\begin{prop} The twisted Dolbeault dga is a $($nonunital\thinspace$)$ really curved dga.  It forms a curved $($but not really curved\thinspace$)$ dga if and only if $\dbar B=0$.
\end{prop}
\begin{proof}  The proof is the same as the proof of Proposition \ref{P:twistedDRalg}. \end{proof}

\section{Bornological algebra}\label{S:bornoalg}

Having constructed nonunital curved dga's associated to gerbes with connection (that is, the twisted de Rham and Dolbeault dga's), we want to describe modules over them and Morita equivalences between them.  This is the first step towards the construction of the dg-category $\psa$ over a nonunital curved dga, which is our main tool for describing dualities.  In this section we begin the project by describing modules and Morita equivalences for arbitrary nonunital algebras.

Some analysis is necessary to describe Morita equivalences between nonunital algebras.  This is because Morita equivalence bi-modules for nonunital algebras are non-finitely generated (see e.g. Proposition \ref{P:sweet}), which necessitates the use of completed tensor products.

We will use bornological analysis as opposed to topological analysis.  Roughly speaking, this means that we shift emphasis from open sets and continuous maps to bounded (\emph{born\'ee} in French) sets and bounded maps.  The category of bornological vector spaces (or more generally bornological $A$-modules) has been found to be the correct framework for homological algebra in the functional analytic context \cite {H},\cite{M1}, \cite{M2}, due to the fact that this category is complete and cocomplete and because there is a tensor product which is adjoint to a certain internal hom.  Further, many algebras arising in noncommutative geometry, for example the smooth compactly supported convolution algebra of a Lie groupoid, have nice properties when viewed as bornological spaces.

We begin with a brief review of bornological analysis.  Here is a short list of references for the reader who is interested in more detail:  a history of the study of bornological spaces, as well as proofs of many of the basic results in the field, can be found in \cite{HN1} and \cite{HN2}; several important results about bornological algebras are proven in \cite{H}; homological algebra relative to bornological spaces is developed in the work of Meyer, for example in \cite{M1} and \cite{M2}.

\subsection{Definitions for bornological analysis}
A \textbf{bornology} on a set $X$ is a collection $\Bc$ of subsets of $X$ that contains the singletons and is closed under taking finite unions and subsets.  The elements of $\Bc$ are called the \textbf{bounded} subsets.
Let $\kb$ be the real or complex numbers.  A $\kb$-vector space equipped with a bornology which is closed under homothetie and finite sums is called a \textbf{bornological $\kb$-vector space}.
A subset $S\subset V$ is called \textbf{circled} if $\lambda v\in S$ whenever $v\in S$, $\lambda\in\kb$ and $|\lambda |\leq 1$.  The circled convex hull $S^\circ$ of any subset $S$ is by definition
\[
S^\circ:=\{\lambda\cdot(tv+(1-t)w)\;|\;\lambda\in \kb,\ |\lambda |\leq 1,\
t\in[0,1],\ v,w\in S\}
\]
A bornology is called \textbf{convex} if for
every bounded set $S\in\Bc$, we have $S^\circ\in\Bc$ as well.  If $S\in\Bc$ is already circled and convex, it is called a \textbf{disc}.

The linear span $V_S$ of a disc S has a unique semi-norm whose unit ball is $S$; $S$ is called \textbf{norming} if this semi-norm is a norm, and \textbf{completant} if $V_S$ is complete with respect to
this norm.  When $S$ is a completant disc, $V_S$ is a Banach space and
is in particular Hausdorff.    Call $V$ a \textbf{complete bornological
vector space (complete bvs)} if every bounded set is contained in a bounded
completant disc.  Note that the bornology of a complete bvs $V$ is generated by the directed set $\Bc_c$ of completant discs, and we may in fact identify $V$ with the inductive limit of Banach spaces
\[ V\isom \lim_{S\in\Bc_c} V_S. \]
The notions of convergence,
Cauchyness, etc... of a net $\{ x_\alpha\!\}$ can be defined within
bornology, and for a complete bvs these are equivalent to the
usual notions, taken in any Banach space $V_S\subset V$ in which the net will eventually be confined.
In particular the closure of any set is defined.

A linear map between two bornological spaces $V$ and $W$ is called \textbf{bounded} if it sends bounded sets to bounded sets.  $\Hom(V,W)$ has a natural bornology called the \textbf{equibounded bornology}, whose bounded subsets are those $S\subset\Hom(V,W)$ such that $S(U)$ is bounded in $W$ for all bounded $U\subset V$.
We write $\Lc(V,W)$ for this bornological vector space.  It can be shown that $\Lc(V,W)$ is
convex (resp. complete) whenever $W$ is convex (resp. complete).

The \textbf{complete tensor product} of two complete bvs's
$V$ and $W$ is a complete bvs $V\tens W$
equipped with a bounded map $V\times W\to V\tens W$ which induces
an isomorphism $\Lc_{Bilinear}(V\times W,X)\isom\Lc(V\tens W,X)$ for all X,
which is universal in the appropriate sense.  It follows from this definition that
\[ \Lc(V\tens W,X)\simeq \Lc(V,\Lc(W,X)). \]
The space $V\tens W$ always exists (see \cite{HN2}).  It is realized concretely as the inductive limit
$V\tens W\isom \lim V_S\otimes_\pi W_T$ where $S$ and $T$ run over the respective directed sets of bounded complete discs in $V$ and $W$, and $\otimes_\pi$ is the projective tensor product of Banach spaces. Let $S$ and $T$ be bounded discs in $V$ and $W$ and let $S\tens T$ denote the smallest complete disc in $V_S\otimes_\pi W_T$ containing the set $\{\ s\tens t | s\in S,\ t\in T \}$.  Such sets form a basis for the bornology of $V\tens W$.

We write \textbf{Bor} for the category of bornological $\kb$-linear spaces with hom sets $\Hom(V,W):=\Lc(V,W).$

\subsection{Examples of bornological spaces}

There are two natural convex bornologies that one can attach to a locally convex topological vector space.  The first is the \textbf{von Neumann bornology}, whose bounded sets are those absorbed by each neighborhood of the origin, and the second is the \textbf{precompact bornology}, whose bounded sets are those which can be covered by a finite union of translates of any neighborhood of the origin.  These bornologies are both complete when the topological vector space is complete.

Any vector space admits the \textbf{fine} bornology, whose bounded sets are those contained in the circled convex hull of some finite set of points.  This is the smallest convex bornology since it is generated by the circled convex hulls of one point sets, and thus every linear map from a fine space is bounded.  Also the tensor product of two spaces, one of which is fine, agrees with the algebraic tensor product.  In particular, endowing a vector space with the fine bornology gives a fully faithful functor
\[ \textrm{ Vector spaces }\tof{fine}\bor \]
 that respects tensor products.  A fine space is complete, and if it happens to be an algebra then it is a complete bornological algebra since the multiplication is automatically bounded.

An important example of a complete bornological vector
space (which has some unpleasant properties as a topological vector space) is the
space of smooth compactly supported functions on a noncompact
manifold.  The bornology is given as follows.  Let $\{K_n\}$ be an increasing sequence of compact sets exhausting a manifold $M$. Then we can write
\begin{equation}\label{E:seminorms}
C^\infty_c(M)=\lim_{n\to \infty}C^\infty_{K_n}(M)
\end{equation}
where $C^\infty_{K_n}(M)$ is the Fr\'echet space of $C^\infty$ functions on $M$ with support in $K_n$. Thus $C^\infty_c(M)$ a \textbf{Limit of Fr\'echet space (LF-space)}.  By definition an LF-space is a locally convex topological vector space which is a countable union $V=\cup V_n$ of increasing subspaces which are Fr\'echet for the subspace topology and for which the topology on $V$ is the finest locally convex topology allowing the inclusions $V_n\hookrightarrow V$ to be continuous.  One can verify that each bounded set in the von Neumann or precompact bornology on an LF space is contained in one of the (complete) subspaces $V_n$ (see \cite{Tr} Chapter 14-15); this implies that these bornologies are complete.

\subsection{Bornological algebras and modules}

A \textbf{bornological algebra} is a bornological vector space with an associative, bounded multiplication.   A \textbf{bornological $A$-module} is a bornological space with a bounded action of $A$.

The vector space of bounded $A$-linear maps between two
$A$-modules has a natural bornology, the \textbf{equibounded}
bornology, in which a collection of maps $S\subset\Hom_A(M,N)$ is
bounded if for every bounded set $U\subset M$, $S(U)$ is bounded
in $N$.  We denote this bornological space $\LA(M,N)$ and note
that it is convex (resp. complete) whenever $N$ is convex (resp.
complete).

The \textbf{$A$-balanced tensor product} of a right module
$M$ and left module $N$, when $A,M$, and $N$ are complete, is the
cokernel of the map
\[ M\tens A \tens N\to M\tens N;\qquad m\tens a\tens n\mapsto ma\tens n-m\tens an \]
where cokernel is defined as the range modulo the closure of the
image.  We write $M\tensA N$ for this (complete) space.

\begin{example}\cite{M1}\label{ex:LF algebra}  Suppose an LF-space $A$ has an associative multiplication that is separately continuous.  Separate continuity on a Fr\'echet space implies joint continuity so the multiplication is jointly continuous when restricted to Fr\'echet subspaces.  But a bounded set in $A\times A$ for the von Neumann or precompact bornology is contained in a Fr\'echet subspace, so multiplication is jointly continuous on bounded sets and therefore bounded. Thus an LF space with associative separately continuous multiplication, called an \textbf{LF algebra}, is a complete bornological algebra for the precompact and von Neumann bornologies.
\end{example}
\textbf{From now on all bornological spaces, including bornological algebras and modules, will be complete unless otherwise stated.}
\[\textrm{Let us set } A^+=
\begin{cases}
& \text{ the \textbf{unitalization} of $A$, when $A$ is nonunital. }\\
      & \text{ $A$, when $A$ is unital. }
\end{cases}
\]
The \textbf{unitalization} is $A\oplus\kb$ as a bornological vector
space, and its elements will be written $\{(a,\lambda)\:|\:a\in
A,\:\lambda\in\kb\}$.

We now recall two properties, \textbf{quasi-unitality} and \textbf{multiplicative convexity}, that a bornological algebra should have in order for its module categories to be well behaved.
\begin{defn}\cite{M2}  A bornological algebra $A$ has an  \textbf{approximate identity} if for any bounded subset $S\subset A$ there is a sequence $\{\mu_n\}\subset A$ (depending on $S$) such that $\mu_nx$ and $x\mu_n$ converge uniformly to $x$ for $x\in S$.  (That means uniformly in a Banach space $A_T$, where $T$ is any bounded complete disc containing $S$.) An algebra is said to be \textbf{quasi-unital} if it has an approximate identity and if furthermore there are left and right $A^+$ module maps
\[ A\to A^+\tens A \text{ and } A\to A\tens A^+  \]
which are sections of the multiplication maps.  (A left or right section exists if and only if $A$ is projective (see below) as a left or right $A^+-$module.)
\end{defn}
In practice there may be a sequence $\{\mu_n\}$ as in the previous definition which does not depend on the bounded subset $S$.  In this case we will say $\{\mu_n\}$ is an approximate identity for $A$.  Then sequences which do depend on $S$ might aptly be called local approximate identities, though we will not need the terminology.

One reason quasi-unitality is important is that if an algebra $A$ is quasi-unital and $M$ is any $A$-module, then the natural map $M\tens_A A\to M$ is injective (\cite{M2} Lemma 15), so that there are plenty of $A$-modules satisfying $M\tens_AA\simeq M$.

We write $MA$ for the image of $M\tens_A A\to M$.
\begin{defn} A bornological algebra is called \textbf{multiplicatively convex} (another name in use is {\em locally multiplicative}) when each bounded set is absorbed by a bounded disc that is closed under multiplication.  Equivalently, $A$ is multiplicatively convex if the spectral radius of every bounded set is finite.  The \textbf{spectral radius} $\rad(S)$ of a bounded set $S$ is the smallest number $t$ such that the \textbf{multiplicative closure} $\bigcup_{n\geq 1}(t^{-1}S)^n$ of $t^{-1}S$ is bounded, or infinity if no such $t$ exists.
\end{defn}
Here are two lemmas concerning multiplicative convexity for bornologies associated to LF-algebras.  They will be used in Section \ref{S:bornoprops}.
\begin{lem}\label{L:LFborno1} \textup{(\cite{Pus} Lemma 1.13)} Let $A$ be a Fr\'echet algebra.  Then the following are equivalent:
\begin{enumerate}
\item There exists a neighborhood $U$ of  $\ 0$ $($called a \textbf{small} neighborhood\thinspace$)$ such that each compact subset of $U$ has precompact multiplicative closure.
\item Each null sequence in $A$ has an end with precompact multiplicative closure.
\end{enumerate}
\end{lem}
Such Fr\'echet algebras are called \textbf{admissible}.
\begin{lem}\label{L:LFborno2} Let $A=\lim A_i$ be an LF-algebra for which each Fr\'echet algebra $A_i$ is admissible.  Then $A$ is multiplicatively convex for the precompact bornology.
\end{lem}
\begin{proof} First, for metric spaces (in particular Fr\'echets) a set is precompact if and only if it is covered by finitely many translates of any neighborhood of the origin.  So let $S$ be a bounded set in $A$.  Then $S$ is actually a precompact subset of some $A_i$.  Let $U$ be a small neighborhood of the origin in $A_i$.  Since $S$ is precompact it is covered by finitely many translates of $U$.  We may assume $U$ is a ball $B_r$ centered at the origin of radius $r$ with respect to some translation invariant metric $\rho$ on $A_i$.  We claim there is a $\lambda\in\Rb$ such that $\lambda U$ actually contains $S$.  This follows because if $y\in B_r$ and $x\in A_i$, we have
\[ \rho(0,x+y)=\rho(-x,y)\leq\rho(-x,0)+\rho(0,y) \]
which implies
\[ x+B_r\subset B_{(\rho(0,x)+r)}.\]
But then $\lambda^{-1} S$ is contained in $U$ and thus has precompact multiplicative closure since $U$ is small.  Thus $S$ has finite spectral radius.
\end{proof}

\subsection{Relative homological algebra}

Let $A$ denote a unital bornological algebra and write $\Mod(A)$
for the category of bornological right $A$-modules. Call a
sequence
\begin{equation}\label{E:exact}
0\lra L\lra M\lra N\lra 0
\end{equation}
in $\Mod(A)$ \textbf{exact} if it is algebraically exact and has a
bounded $\kb$-linear splitting.  Then a module $P$ is said to be
\textbf{projective} if the functor $\LA(P,\cdot\ )$ takes exact sequences
in $\Mod(A)$ to exact sequences in $\bor$.  $P$ is projective if
and only if it is a direct summand of a \textbf{(relatively) free
module}, that is a module of the form $\Free(V):=A\tens V$ for
some $V\in\bor$.  We say relatively free because the functor
$\Free$ is left adjoint to the forgetful functor
\[\forget:\Mod(A)\lra\bor, \]
as opposed to the usual forgetful functor to vector spaces.  $\Mod(A)$
is not in general an Abelian category but it still has enough nice
properties:
\begin{prop}\cite{M2} $\Mod(A)$ with the class of exact sequences of the form \eqref{E:exact} is a Quillen exact category.  This category is complete and cocomplete and has enough injectives and enough projectives.
\end{prop}
This means that as long as we work relative to $\bor$, most of the
tools of homological algebra are still available.

When $A$ is nonunital, define $\Mod(A):=\Mod(A^+)$.  This is consistent notation because every action of $A$ extends uniquely to a unital action of $A^+$.  We may use the notations $\Lc_A$ and $\Lc_{A^+}$ as well as $\otimes_A$ and $\otimes_{A^+}$ interchangeably.

It should be noted however, that the free modules in $\Mod(A)$ are of the form $A^+\tens V$.  In particular $A=A\tens\kb$ is not a free $A$-module when $A$ is nonunital, since $\LA(A\tens\kb,A)$ and $\Hom_{\bor}(\kb,A)\isom A$ are not isomorphic (the first one is unital and the second is not).  Thus one should remember that the terms finitely generated, projective and free are defined in $\Mod(A^+)$.  Those modules $M$ over a quasi-unital algebra $A$ that satisfy $M\tensA A\isom M$ are called \textbf{essential} or \textbf{non-degenerate} modules.

 Now we have the terminology to define a Morita equivalence bimodule for nonunital algebras.
\begin{defn}\label{D:moritabimodule} Let $A$ and $B$ be quasi-unital bornological algebras.  Then an essential $(A\!-\! B)-$bimodule $X$ is a \textbf{Morita equivalence bimodule} if there is an essential projective $(B\!-\! A)-$bimodule $Y$ and bornological isomorphisms
\[
A\isom X\tensB Y \:\text{ and }\:
B\isom Y\tensA X.
\]
\end{defn}
\begin{lem}\label{rem:dualmodule} Suppose that $X$ is a Morita equivalence $A-B$ bimodule and let $Y$ be a dual Morita equivalence bimodule.  Then $Y\simeq B X^\vee$, where $X^\vee:=\LA(X,A).$
\end{lem}
\begin{proof}

First, the assumptions imply $\LA(X,A)\cong\LB(Y\tensA X, Y\tensA A)\cong\LB(B,Y)$.  The right hand side is called the roughening of $Y$ in \cite{M2}, and it is shown there (Theorem 22) that the essentialization of the roughening is the essentialization, that is, \linebreak $B\tensB\LB(B,Y)\cong B\tensB Y$.  By assumption, $B\tens B Y\cong Y$, so we have the desired result.
\end{proof}
\begin{prop}\label{P:MoritaInvariantConvexity} Suppose two bornological algebras $A$ and $B$ are Morita equivalent, then if $B$ is multiplicatively convex, so is $A$.
\end{prop}
\begin{proof} Let $X$ be a Morita equivalence $A-B$-bimodule with dual bimodule $Y$.  By Lemma \ref{rem:dualmodule}, we may assume the dual bimodule is $Y\simeq B\tensB X^\vee$ where $X^\vee:=\Lc_A(X,A)$.  Thus  $X\tensB X^\vee\simeq X\tensB B\tensB X^\vee\simeq A$ and the multiplication in $A\simeq X\tensB X^\vee$ is $x_1\tens\phi_1\circ x_2\tens\phi_2:=x_1\phi_1(x_2)\tens\phi_2$.

We will show that every bounded subset of $X\tensB X^\vee$ has finite spectral radius.  Indeed, since $X\tensB X^\vee$ has the quotient bornology from $X\tens_\kb X^\vee$, every bounded subset is contained in the image in $X\tensB X^\vee$ of a set of the form $S\tens_\kb T^\vee$ where $S$ is a bounded subset of $X$ and $T$ is bounded in the equibounded bornology of $X^\vee$.  But one can check that $(S\tensB T^\vee)^n\subset S(T^\vee(S))^{n-1}\tensB T^\vee$, which implies $\rad(S\tensB T^\vee)\leq \rad(T^\vee(S))<\infty$.
\end{proof}

Note that whenever $X$ is finitely generated, the bornological tensor product equals the algebraic one, so we are reduced in the unital case to previous definitions which make no mention of bornology.

\section{Homotopy nuclear modules}

In this section we solve the following problems:  What is a useful notion of a finitely generated projective module over a nonunital algebra?   And if there are useful ``finite projective'' type conditions for modules over nonunital algebras, are these conditions preserved under Morita equivalence?

Here are two reasons that these are not trivial problems:
\begin{enumerate}
\item In general Morita equivalence bimodules are non-finitely generated, so there is no obvious reason that they should take finitely generated objects to finitely generated objects.
\item The truly ``noncompact'' algebras (those which are not Morita equivalent to any unital algebra) may have no essential finitely generated projective modules at all (see Proposition \ref{P:sweet}).  So the classical notion is not a useful notion.
\end{enumerate}

First we will provide an answer that works for algebras which are Morita equivalent to unital ones.  Then, adapting an idea of Quillen's (see Subsection \ref{S:Quillen}), we proceed with a solution which works for the ``noncompact'' algebras.

\subsection{Finite modules and Morita equivalence}
\begin{defn}
Call a module over a quasi-unital algebra \textbf{finite} if it is finitely generated, projective and essential.
\end{defn}
We will show that a Morita equivalence bimodule takes finite modules to finite modules.
It is easy to see that a Morita $(A-B)$-bimodule $X$ induces an equivalence between the categories of essential $A$-modules and essential $B$-modules, this is just because $X$ is itself essential.  Also, it is not hard to show that under this equivalence projectives are taken to projectives.  As remarked above, however, it is not obvious that the finite modules are taken to finite modules under this equivalence, since $X$ itself is not finitely generated.  Nonetheless, arguments using nuclearity will show that it is true.

\begin{defn} Let $M$ and $N$ be modules over a unital bornological algebra $B$, and let $A$ be an ideal in $B$.  A map $f\in\LB(M,N)$ is called \textbf{$A$-nuclear} (or simply \textbf{nuclear} when $A=B$) if it lies in the image of the natural map
\begin{equation}\label{E:nuclear}
N\otimes_B\Lc_B(M,A)\lra\Lc_B(M,N);\qquad (n\otimes
\phi)(m):=n\cdot\phi(m)
\end{equation}
\end{defn}

The following nontrivial proposition characterizes finitely generated projective modules over a bornological algebra $A$ in terms of nuclearity.  The conditions that $A$ be complete and multiplicatively
convex are indispensable.
\begin{prop}\cite{H}\label{P:nuclear}
Let $A$ be a multiplicatively convex unital bornological algebra and $M$ an $A$-module.
Then $M$ is finitely generated and projective if and only if $1_M\in\LA(M,M)$
is nuclear.
\begin{flushright}
    $\square$
\end{flushright}
\end{prop}
Note that even though the above proposition is for unital algebras, it also applies in the nonunital case because we have defined an $A$-module to be a module over the unitalization of $A$.

\begin{prop}\label{P:MoritaUnitalCase} A Morita equivalence $A-B$ bimodule $X$ takes finite modules to finite modules.
\end{prop}
\begin{proof}
Let $M$ be a finite $A$-module.  Since
$X$ is an essential $B$-module, so is\linebreak $M\tensA X=:N$.  Now examine the following commutative diagram
$$\label{D:diagram1}
\xymatrix{ M\tensA X\tensB Y\tensA\LA(M,A^+) \ar[r] \ar[d] & M\tensA\LA(M,A^+) \ar[r] & \LA(M,M) \ar[d] \\
 N\tensB Y\tensA\LB(N,A^+\tensA X) \ar[r] &
 N\tensB\LB(N,B) \ar[r] & \LB(N,N)
}
$$
where the vertical maps are induced from $\_\tensA X$, the upper left map is induced from the map $M\tensA (X\tensB Y)\isom M\tensA A\to M$, and the lower left map is induced from the natural map $Y\tensA\LB(N,A^+\tensA~X)\to\LB(N,Y\tensA A^+\tensA X)=\LB(N,B)$.

The upper left map is an isomorphism by essentialness of $M$ and the upper right map is an isomorphism since $M$ is finitely generated and projective.  But then the identity map $id_N\in\LB(N,N)$ is the image of an element from the upper left corner.  But then by commutativity of the diagram, $id_N$ factors through $N\tensB\LB(N,B)$.  By Proposition \ref{P:nuclear}, $N=M\tensA X$ is a finitely generated projective $B^+$-module.
\end{proof}
\begin{rem}\label{R:imageIsEssential} In the proof above we used the fact that since $N$ is an essential $B$-module,
\[ \LB(N,B)=\LB(N,B^+). \]
To see that every $\phi\in\LB(N,B^+)$ indeed has its image in $B$, let $\mu:N\tensB B\to N$ denote the $B$-action, and note that $\phi$ factors as follows:
$$\label{D:diagram1}
\xymatrix{
N\tensB B\ar[r]^{\phi\otimes I}  & B^+\tensB B \ar[d] \\
N \ar[r]^\phi \ar[u]^{\mu^{-1}} & B^+.
}
$$
Since the multiplication map $B^+\tens B\to B^+$ has image in $B$, so does $\phi$.
\end{rem}

If an algebra is unital, then the condition of essentialness is automatically satisfied, so in this case a finite module is just a finitely generated projective module.  Thus Proposition \ref{P:MoritaUnitalCase} shows that for any algebra which is Morita equivalent to a unital algebra, the finite modules are the correct replacement for finitely generated projectives.

On the other hand, for algebras that are not Morita equivalent to a unital algebra, there may be no finite modules at all.  Indeed,
\begin{prop}\label{P:sweet}  Let $A$ be a nonunital algebra and let $X$ be a  finitely generated projective essential module such that the evaluation map \[ X\tens\LA(X,A)\lra A \]
is surjective.  Then $X$ induces a Morita equivalence between $A$ and the unital algebra $\LA(X,X).$
\end{prop}
\begin{proof}
Let $X^\vee:=\LA(X,A)$ and define $B:=\LA(X,X)$.  Here $A$ acts on the right of $X$.  We claim that the pair $(X^\vee,X)$ is a Morita equivalence.  Indeed, $B \simeq X\tensA X^\vee$  since $X$ is finitely generated and projective.  By hypothesis the evaluation map $X^\vee\tens X\to A$ is surjective.  The evaluation descends to a map $X^\vee\tensB X\to A$, and this map can be checked directly to be injective because $X$ and $X^\vee$ are cyclic for $B$ (that is, for any nonzero $x\in X$ and $\xi\in X^\vee$ we have $X=B\cdot x$ and $X^\vee=\xi\cdot B$).  Thus $X^\vee\tensB X\simeq A$.
\end{proof}
Thus for example if $M$ is a noncompact manifold, then $C_c^\infty(M)$ admits no finitely generated projective essential modules.  Indeed, such a module would correspond to a vector bundle and would satisfy the hypotheses of the above proposition, thus inducing a Morita equivalence between $C_c^\infty(M)$ and a unital algebra.  But no such Morita equivalence can exist because a multiplicatively convex unital algebra has compact spectrum, whereas $M$ is noncompact.

Next we describe a weaker version of finiteness that will work whether an algebra is Morita equivalent to a unital algebra or not.

\subsection{Quillen's Method}\label{S:Quillen}

\begin{defn} Let $A$ be a quasi-unital bornological algebra.  A \textbf{homotopy-finite complex (or h-finite complex)} is a complex $M$ of bornological $A$-modules that is both homotopy equivalent to its essentialization $M\tensA A$, and homotopy equivalent to some complex of finitely generated projective $A$-modules. (Remember this means finitely generated projective as an $A^+$-module.)  \textbf{All chain complexes are assumed to be bounded in grading degree unless otherwise stated.}
\end{defn}
The class of algebraic h-finite complexes was introduced by Quillen \cite{Q}, and by characterizing algebraic h-finiteness in terms of nuclearity, this class was shown to be stable under Morita equivalence.  We will work out the necessary modifications for the bornological setting.

First, though, let us point out that there are plenty of h-finite complexes.
\begin{example}\label{E:Psharp}
Let $P$ be an $A$-module, and define
\[ P^\sharp:=(P/PA)\tens_\kb A^+. \]
Consider the diagram
$$
\xymatrix{  P \ar[r] & P/PA \ar[d]\\
  P^\sharp \ar[r] & P^\sharp/P^\sharp A
}
$$
where the vertical arrow is the canonical isomorphism.  If $P$ is finitely generated and projective, this can be filled in to a commutative diagram with a map
\[ \phi:P\lra P^\sharp. \]
Inverting the vertical map, we can get an analogous commutative diagram filled in by a map $\psi:P^\sharp\lra P$.  Then $E:=(P\stackrel{\phi}{\to}P^\sharp)$ is an h-finite complex.  Indeed, $E$ is made of finitely generated projective $A^+$-modules, and furthermore, the inclusion $EA\hookrightarrow E$ is checked to have homotopy inverse given by \[f:=id_E-[\phi,\psi]:E\to EA,\]
so $E$ is homotopy equivalent to its essentialization.

Thus every finitely generated projective module determines an h-finite complex.

This example applied to the commutative setting is interpreted as follows. Suppose that $A$ is the algebra of smooth compactly supported functions on a noncompact manifold M.  Then $P$ corresponds to a rank $r$ vector bundle on $M$ and $P^\sharp$ corresponds to the trivial rank $r$ bundle.  The map $P\lra P^\sharp$, corresponds to a map of vector bundles, which is an isomorphism at infinity.  In other words, $P\lra P^\sharp$ is a bundle with a fixed trivialization in a neighborhood of infinity.
\end{example}

\begin{defn}\label{D:homotopynuclear}  Given two chain complexes $M^\bullet,N^\bullet\in Ch(B)$, the graded maps form a complex of bornological spaces $\LB^\bullet(M,N)$ with differential
\[ d(h)=d_N\circ h-(-1)^ih\circ d_M,\qquad h\in\LB^i(M,N). \]
If $A$ is an ideal in $B$, the \textbf{$A$-nuclear maps} (from $M^\bullet$ to $N^\bullet$) are defined as the image of the morphism
\begin{equation}\label{E:nuclearcx}
N^\bullet\otimes_B\Lc_B^\bullet(M,A)\lra\Lc_B^\bullet(M,N);\quad
(n\otimes \phi)(m):=n\cdot\phi(m).
\end{equation}
Those maps coming from the algebraic tensor product are called \textbf{algebraically nuclear}.
\end{defn}
A morphism of complexes is $A$-nuclear exactly when its graded components are.
\begin{defn}
A morphism $f\in\LB^k(M,N)$ is called \textbf{homotopy-$A$-nuclear} when it is homotopic to an $A$-nuclear map, that is, when there exists an $h\in\LB^{k-1}(M,N)$ and an $A$-nuclear $g$ satisfying $f-g=d(h)$.
\end{defn}
The following two propositions are generalizations to the bornological setting of Propositions 1.1 and Lemma 2.1 of \cite{Q}.  These will be used to characterize h-finite complexes in terms of homotopy nuclearity.
\begin{prop}\label{P:honuclear} Let $B$ be a unital multiplicatively convex bornological algebra.  For a complex $M^\bullet\in\Ch(B)$, the following are equivalent.
\begin{enumerate}
\item $M^\bullet$ is homotopy equivalent to a complex of finitely generated projective $B-$modules.\label{one}
\item $M^\bullet\tensB\LB^\bullet(M,B)\lra\LB^\bullet(M,M)$ is a homotopy equivalence.\label{two}
\item $id_M\in\LB^\bullet(M,M)$ is homotopy equivalent to a nuclear map.\label{three}
\item $id_M\in\LB^\bullet(M,M)$ is homotopy equivalent to an \textbf{algebraically} nuclear map.\label{four}
\item There is a finite complex of finitely generated projective modules $T$ and chain maps $M\stackrel{f}{\to}T\stackrel{g}{\to}M$ such that $gf$ is homotopic to $id_M$.\label{five}
\end{enumerate}
\end{prop}
\begin{proof} \eqref{one}$\Rightarrow$\eqref{two}$\Rightarrow$\eqref{three}.  Clear.

\eqref{three}$\Rightarrow$\eqref{four}.
Suppose
\[ id_M-f=d(h), \text{ for } f\in M\tensB\LB(M,B)\text{ and } h\in\LB(M,M). \]
We may assume $f$ and $h$ are homogeneous of degrees $0$ and $-1$ respectively.  Houzel (\cite{H}) points out that every $f\in M\tensB\LB(M,B)$ can be written as
\[ f=\sum\lambda_nx_n\tens\phi_n\]
where $\{x_n\}\subset M$ and $\{\phi_n\}\subset\LB(M,B)$ are sequences converging to zero and $\{\lambda_n\}\in\ell^1(\kb)$, and that consequently $f$ may be written
\[ f=\eps + F \]
where $F$ is of finite rank (i.e.\! is algebraically nuclear) and $\eps$ is so small that $id_M-\eps$ is invertible.  Here we know that the sum $\sum\eps^n$ converges because $M\tensB\LB(M,B)$ is Morita equivalent to $B$ and is therefore multiplicatively convex by \ref{P:MoritaInvariantConvexity}.  Let $\phi:=(id_M-\eps)^{-1}$.  Then we have
\[ id_M-\phi F=\phi\cdot(id_M-\eps-F)=\phi\cdot(id_M-f)=\phi d(h)=d(\phi h)-d(\phi)h \]
so
\[ id_M - \{ \phi F - d(\phi)h \}=d(\phi h), \]
meaning that the identity is homotopic to $\phi F - d(\phi)h$.  We claim that $\phi F - d(\phi)h$ is actually of finite rank.  First note that a composition of maps is finite rank if one of the maps is, so in particular $\phi F$ is finite rank.  But also $d(\phi)$ and thus $d(\phi)h$ is finite rank.  To see this, note that
\[ 0=d^2(h)=d(id_M-f)=d(\phi^{-1}-F) \]
so that $d(\phi^{-1})=d(F)$.  Then $d(F)=d(\sum_{n=1\dots N}\lambda_nx_n\tens\phi_n)$ is clearly finite rank. Finally, $d(\phi)$ is a composition with one factor being finite rank: $d(\phi)=\phi d(\phi^{-1})\phi$, so $d(\phi)$ is finite rank.  Thus homotopy nuclear implies algebraically homotopy nuclear, as desired.

\eqref{four}$\Rightarrow$\eqref{five}.  The $f$ and $g$ and the homotopy $h$ making $id_M-d(h)=gf$ that are constructed in Proposition 1.1 of \cite{Q} are obviously bounded maps, therefore apply to the bornological setting.

\eqref{five}$\Rightarrow$\eqref{one}.  Proposition (3.2) of \cite{Ran} shows this result algebraically by constructing a new complex $T'$ out of finite sums of shifts of $T$, and new chain maps
\[ M\stackrel{f'}{\to}T'\stackrel{g'}{\to}M, \]
and a graded map
\[ T'^\bullet\stackrel{h'}{\to} T'^{\bullet-1} \]
satisfying
\begin{enumerate}
\item[(i)] $g'f'=gf=id_M-d(h)$
\item[(ii)] $f'g'=1_{T'}-d(h').$
\end{enumerate}
The three maps $f'$, $g'$ and $h'$ are finite sums of finite compositions of shifts of $f$, $g$ and $h$.  Thus
$M$ is homotopy equivalent to the finite complex of finitely generated projectives $T'$, as desired, and since $f$, $g$, and $h$ are bounded maps so are $f'$, $g'$ and $h'$.
\end{proof}
\begin{prop}\label{P:hoAnuclear}
Let $A$ be a quasi-unital ideal in a unital bornological algebra $B$.  Then for a complex $M$ of projective $B$-modules, the following are equivalent.
\begin{enumerate}
\item $M^\bullet\tensB A\lra M^\bullet$ is a homotopy equivalence.
\item $id_M\in\LB^\bullet(M,M)$ is homotopic to a map with image in $MA$.
\end{enumerate}
\end{prop}
\begin{proof} Suppose (1) holds, then choose a homotopy inverse $f$ for the inclusion $MA\simeq M\tensB A\incl M$.  By definition there is an $h$ such that $id_M-f=d(h)$, so (2) is true.  Now suppose (2) is true.  Choose an $f$ with image in $MA$ and an $h$ so that $id_M-f=d(h)$ holds.  Then $f$ is a homotopy inverse for the inclusion of (1).  One need only note that since $h$ is $A$-linear it takes $MA$ into itself and thus is a homotopy operator for $MA$ as well as for $M$.
\end{proof}
\begin{cor}\label{C:hfinite} Let $A$ be a quasi-unital multiplicatively convex algebra, then a complex $M^\bullet\in\Ch(A)$ is h-finite if and only if the identity map $id_M\in\LA^\bullet(M,M)$ is homotopic to an A-nuclear map.
\end{cor}
\begin{proof}  This is essentially the same as the proof of Theorem 2.4 of (\cite{Q}).  Suppose $id_M$ is homotopy-nuclear.  Then by Proposition \ref{P:honuclear} $M$ is homotopy equivalent to a finite projective complex, and by Proposition \ref{P:hoAnuclear} $M$ is homotopy-essential, thus $M$ is h-finite.  Conversely, suppose $M$ is h-finite.  Then in the following commutative square of canonical maps
$$
\xymatrix{  M^\bullet\tensA A\tensA\LA^\bullet(M,A^+)\ar[r]\ar[d] &  M^\bullet\tensA\LA^\bullet(M,A^+)\ar[d]\\
  M^\bullet\tensA\LA^\bullet(M,A) \ar[r] & \LA^\bullet(M,M)
}
$$
the top map is a homotopy equivalence since $M$ is homotopy essential, and the right map is a homotopy equivalence by Proposition \ref{P:honuclear}.  But then the lower left route is a homotopy equivalence too, and in particular $id_M$ is homotopic to a map coming from $M\tensA\LA(M,A)$.
\end{proof}

Write \textbf{hFin(A)} for the category whose objects are h-finite complexes of $A$-modules and whose morphisms from $M$ to $N$ are $\Lc^\bullet(M,N)$.  This category is enriched in chain complexes of bornological vector spaces, so it is a dg-category with an extra bornological bit, which we call a \textbf{dgb-category}.  A functor between dgb-categories will mean for us a dgb-enriched functor, that is a functor whose maps on hom sets are dgb-morphisms.

Let $\Ho(\hFin(A))$ be the category with the same objects as $\hFin(A)$ and with morphisms $\Hom(M,N):=H^0(\LA^\bullet(M,N)).$  A functor $F:\Cc\to\Dc$ of dgb-categories is called a \textbf{quasi-equivalence} if
\begin{itemize}
\item[(i)] For all $M,N\in\Cc$, $F_{M,N}:\Hom_\Cc(M,N)\to\Hom_\Dc(FM,FN)$ is a quasi-isomorphism, and \item[(ii)] $\Ho(F):\Ho(\Cc)\to\Ho(\Dc)$ is an equivalence of categories.
\end{itemize}
Now given a Morita $(A-B)-$bimodule $X$, we have the functor
\[ (\_\_\;)\tensA X:\Mod(A)\to\Mod(B). \]
We want to see what happens when this is applied to $\hFin(A).$  One cannot hope to get an equivalence of categories between $\hFin(A)$ and $\hFin(B)$ because anything in the image of $(\_\_\;)\tensA X$ is essential, while there are h-finite complexes that are not essential (such as those in example \ref{E:Psharp}).  However, $X$ does induce a quasi-equivalence (see Theorem \ref{T:dgMorita}).  Typically one is only interested in the quasi-equivalence type of a dg-category anyway, so this is a satisfactory result. On the other hand, a Morita bimodule is more than one needs to get a quasi-equivalence, which motivates the following useful notion:
\begin{defn}
A complex of $(A\!-\! B)-$bimodules $X^\bullet$ is a \textbf{homotopy-Morita bimodule} (or h-Morita bimodule) if there is a dual $(B\!-\! A)-$bimodule $Y^\bullet$ and homotopy equivalences
\[
A\approx X^\bullet\tensB Y^\bullet \:\text{ and }\:
B\approx Y^\bullet\tensA X^\bullet.
\]
\end{defn}

\begin{thm}\label{T:dgMorita} Let $X^\bullet$ be an h-Morita $(A-B)-$bimodule.  Then when $A$ and $B$ are quasi-unital $X^\bullet$ induces a dgb-quasi-equivalence
\[ (\_\_\;)\tensA X^\bullet:\hFin(A)\to\hFin(B) \]
\end{thm}
\begin{proof} Let $Y^\bullet$ be the dual module for $X^\bullet$, $M^\bullet\in\hFin(A)$, and set $N^\bullet:=M^\bullet\tensA X^\bullet$.
We have a commutative diagram analogous to the diagram of
Proposition \ref{P:MoritaUnitalCase},
$$
\xymatrix{ M^\bullet\tensA X^\bullet\tensB Y^\bullet\tensA\LA^\bullet(M,A^+) \ar[r] \ar[d] &
 M^\bullet\tensA\LA^\bullet(M,A^+) \ar[r] & \LA^\bullet(M,M) \ar[d] \\
 N^\bullet\tensB Y^\bullet\tensA\LB^\bullet(N,A^+\tensA X) \ar[r] &
 N^\bullet\tensB\LB^\bullet(N,B) \ar[r] & \LB^\bullet(N,N)
}
$$
only this time it is a diagram of chain complexes and chain
morphisms and the top arrows are homotopy
equivalences instead of isomorphisms.  This implies that the
identity $1_N\in\LB^\bullet(N,N)$ is homotopy equivalent to a map
$f$ coming from $M^\bullet\tensA X^\bullet\tensB
Y^\bullet\tensA\LA^\bullet(M,A^+)$.  Then by commutativity of the
diagram, $f$ factors though $N^\bullet\tensB\LB^\bullet(N,B)$.
By Corollary \ref{C:hfinite}, $N^\bullet$ is h-finite.  Thus $X$ takes h-finites to h-finites, and the reverse map induced by $Y$ is a quasi-inverse.
\end{proof}

\section{The homotopy nuclear perfect category}\label{S:psafinally}

Having described a workable replacement for the finitely generated projective modules over a nonunital bornological algebra, that is the h-finite complexes, we can finally return to nonunital curved dga's and their modules.  These modules will be the analogue of objects of the dg-category $\psa$ over a unital curved dga as defined in \cite{Bl1}, now adapted to the nonunital bornological setting.  They will describe, for example, modules over the twisted De Rham and Dolbeault dga's.

From now on all algebras and modules will be assumed bornological and all maps between them will be bounded maps.  For the sake of brevity, however, we will usually not mention the word bornological anymore.

\begin{defn}\label{D:curveddgastuff} A \textbf{bornological curved dga} is a triple $\As=(\Adot,d,c)$, where $\Adot$ is a quasi-unital multiplicatively convex graded bornological algebra and
\[d:\A^\bullet\to \A^{\bullet+1}\]
is a bounded map satisfying
\begin{itemize}
\item[(i)] $d(ab) = d(a)b + (-1)^{|a|}ad(b)$
\item[(ii)]$d^2(a)=[c,a];\quad c\in M(\A)^2$
\item[(iii)] $dc=0$.
\end{itemize}
Here $M(\A)^\bullet$ is the multiplier algebra of $\Adot$.  The derivation $d$ is defined on $M(\A)^\bullet$ by the formula $(dc)a:=d(ca)-cda$ for $a\in\Adot$.  As usual, we write $\A$ for the degree zero component $\A^0$.  If condition (iii) is not satisfied we call $\As$ a \textbf{really curved bornological dga}.

The \textbf{tensor product} of two curved dga's is
\[ \As\tens\Bs := (\Adot\tens\Bdot,d_\Ac\tens 1+1\tens d_\Bc,c_\Ac\tens 1+1\tens c_\Bc) \]
and the \textbf{opposite} of a curved dga is
\[ \As^{\opp} := ({\Adot}^{\opp},d_\Ac,-c_\Ac). \]
A \textbf{morphism} between two curved dga's $\As=\Adc$ and $\Bs=\Bdc$ is a pair $(f,\omega)$ where $f:\Adot\to\Bdot$ is a morphism of graded algebras, $\omega\in M(\Bc)^1$, and the pair satisfies
\begin{itemize}
\item[(i)]$ f(d_\Ac a)=d_\Bc(f(a))+[\omega,f(a)]$
\item[(ii)]$f(c_\Ac)=c_\Bc+d_\Bc(\omega)+\omega^2\in M(\Bc)^2$.
\end{itemize}
\end{defn}
\begin{defn}\label{D:Zconnection}
A \textbf{(right) module} over a curved dga $\As$ is a pair $E=(E^\bullet,d^E)$ in which $E^\bullet$ is a complex of essential right $\A$-modules and $d^E$ is a \textbf{$\Zb$-connection}.  A $\Zb$-connection is a $k$-linear map
\[d^E: E^\bullet\tensAc\A^{\bullet}\to E^\bullet\tensAc\A^{\bullet}  \]
of total degree one, satisfying
\begin{itemize}
\item[(i)] $d^E(e\tensAc \omega)=d^E(e)\omega+(-1)^{|e|} e\tensAc d\omega$
\item[(ii)]$d^E\circ d^E(e\tensAc a)=-e\tensAc ac$
\end{itemize}
Condition (ii) expresses the requirement that $E$ have \textbf{curvature} $c$; the minus sign occurs because we are dealing with right modules.
\end{defn}
Note that $d^E$ is determined by its values on $E^\bullet$, and decomposes as a sum $d^E=\sum_{k\geq 0}d^E_k$ where
\[ d^E_k:E^\bullet\to E^{\bullet+1-k}\tensAc\A^k \]
and the $d^E_k$ are $\Adot-$linear for all $k\neq 1$.  Because $(d^E_0)^2=0$, $(E^\bullet,d^E_0)$ is a chain complex over $\A$.
For any such module $E$, the script font $\Ec$ will denote the $\Adot-$module
\[  \Ec^\bullet:=\tot(E^\bullet\tensAc\A^{\bullet}).  \]
The morphisms between two $\As$-modules $E=(E^\bullet,d^E)$ and $F=(F^\bullet,d^F)$ is the complex of $\Adot$-(graded)-linear maps from $\Ec$ to $\Fc$,
\[
\Hom_{\negthickspace\As}^q(E,F)=\LAdot^q(\Ec,\Fc):=\{\phi\in\Lc(\Ec^\bullet,\Fc^{\bullet+q})\ |\ \phi(ea)=\phi(e)a,\quad a\in\A^\bullet\}
\]
with the differential $d^{EF}$ defined in the standard way,
\[
d^{EF}(\phi)(e)= d^F(\phi(e))-(-1)^{\abs{\phi}}\phi(d^Ee).
\]
The curvings on $E$ and $F$ cancel each other so $d^{EF}$ does indeed square to zero.  These modules and their morphism complexes form a dgb-category, which we denote $\mathbf{Mod(\As)}$.

In any dg-category, two morphisms
\[  f,g\in Hom(E,F) \]
are called homotopic when they are closed and their difference is a boundary.  Then the two objects $E$ and $F$ are called homotopy equivalent when there are closed degree zero morphisms
\[ f:E\rightleftarrows F:g \]
that are mutual inverses up to homotopy.

So even though the objects of $\Mod(\As)$ are not complexes, homotopy equivalence makes sense.  Every object of $\Mod(\As)$ does have an underlying chain complex, however.  The following lemma implies that whenever two objects are homotopy equivalent in $\Mod(\As)$, then their underlying complexes are homotopy equivalent in $Ch(\A)$.

\begin{lem} Let $U:\Mod(\As)\lra Ch(\A)$ be the ``underlying chain complex'' functor which sends an $\As$-module $(E^\bullet,d^E)$ to the chain complex of $\Ac$-modules $(E^\bullet,d^E_0)$, and sends a morphism $\phi\in\Hom^k(E,F)$, which is determined by its components
\[  \phi_j:E^{\bullet-j}\to F^{\bullet+k}\tens A^j, \]
to $\phi_0: E^\bullet\to F^\bullet.$  Then $U$ is a dgb-functor.
\end{lem}
\begin{proof}  This is fairly obvious, but let us check that $\Hom^\bullet(E,F)\to\Hom^\bullet(U(E),U(F))$ is a chain morphism, meaning that for $\phi\in \Hom^k(E,F)$, $d(\phi_0)=(d\phi)_0$:
\[ d(\phi_0):=d^F_0\circ\phi_0-(-1)^k\phi_0\circ d^E_0=(d^F\circ\phi-(-1)^k\phi\circ d^E)_0=(d\phi)_0. \]
\end{proof}
\begin{cor}\label{C:underlying} A homotopy equivalence in $\Mod(\As)$ induces a homotopy equivalence of the underlying complexes of $\Ac$-modules.
\begin{flushright}
$\square$
\end{flushright}
\end{cor}

We can now define h-finite objects in $\Mod(\As)$.
\begin{defn} An $\As$-module $(E^\bullet,d^E)$ is called \textbf{h-nuclear} if its underlying complex is homotopy equivalent to a complex of $\A-$modules $P=(P^\bullet,d^P)$ with each $P^k$ finitely generated projective, or equivalently if the underlying complex of $(E^\bullet,d^E)$ is h-nuclear in $Ch(\Ac)$.
\end{defn}
One might think it is more convenient to work with $\As$-modules whose underlying complex is just h-finite, that is we might try to relax the essentialness condition to homotopy essentialness.  But for non-essential modules there are some technical difficulties that arise.  For example one would then have to be careful to define the $\Zb$-connection on $E^\bullet$ rather than the space $E^\bullet\tensA\A$.  To avoid such issues we always require that modules be essential.
Remember our convention is that all chain complexes are bounded in grading degree.
\begin{defn}  Let $\As$ be a bornological curved dga.  The \textbf{homotopy-perfect category (or h-perfect category)} over $\As$ is the full sub-dgb-category
\[ \hPa\subset\Mod(\As) \]
of h-nuclear $\As-$modules.
\end{defn}
The h-perfect category always contains as a full subcategory the \textbf{perfect category} $\psas$ whose objects are the \textbf{finite} $\As$-modules, that is those $(E^\bullet,d^E)$ for which each $E^k$ is essential and finitely generated projective over $\Ac$.

The following lemma implies that whenever $\Adot$ is unital, the inclusion $\psas\hookrightarrow\hPa$
is a dgb-quasi-equivalence, so we have defined an honest generalization to the nonunital case of the perfect category.

\begin{lem}\label{L:grauert}\cite{Bl1}  Let $\As$ be a unital curved dga, and $E\in\Mod(\As)$.  Suppose the underlying complex $(E^\bullet,d^E_0)$ of $E$ is homotopy equivalent to a complex $(P^\bullet,d)$ that is bounded in grading degree and with each $P^k$ a finitely generated projective $\A$-module.  Then $(P^\bullet,d)$ can be extended to an $\As$-module $(P^\bullet,d^P)$ which is homotopy equivalent in $\Mod(\As)$ to $E$ and such that $d=d^P_0$.
\end{lem}

The next goal is to define the bimodules that induce morphisms between curved dga's (or their associated h-perfect categories), and to single out a class of equivalence bimodules that induce equivalences of h-perfect categories.
\begin{defn}\label{bimod}
Given a pair $(\As,\Bs)$ of curved dga's, an \textbf{$(\As$-$\Bs)$-twisted bimodule} $X=(X^\bullet,d^X)$ is a graded right essential $\Bc$-module $\Xdot$ equipped with the following structure:
\begin{enumerate}
\item A left action of $\Adot$ on $X^\bullet\tensBc\Bc^{\bullet}$ satisfying \\
\quad(i) $\:a\cdot (x\otimes b_1b_2)=(a\cdot(x\otimes b_1))\cdot b_2\quad$.
\item A degree one map $d^X:\Xdot\tensBc\Bc^{\bullet}\to\Xdot\tensBc\Bc^{\bullet}$ satisfying\\
\quad(i) $d^X(a\cdot(x\otimes b_1)b_2)=d_\A a\cdot(x\otimes b_1b_2)+(-1)^{|a|}a\cdot d^X(x\otimes b_1)b_2$\\
         $+(-1)^{|a(x\otimes b_1)|}a\cdot(x\otimes b_1)d_\Bc b_2$\\
\quad(ii) $d^X\circ d^X(x\otimes b)=c_\A\cdot (x\otimes b)-(x\otimes b)\cdot c_\Bc$\\
          for $a\in\Ac^{\bullet}$, $x\in X^\bullet$ and $b_1,b_2\in\Bc^{\bullet}$.
\end{enumerate}
Note that $(\A,d_\A)$ is an $(\As$-$\As)$-twisted bimodule.
\end{defn}
Twisted bimodules are designed to be the Morita type morphisms, in the sense that an $(\As$-$\Bs)$-twisted bimodule $X$ induces a dg-functor
\[ X_*:\Mod(\As)\lra\Mod(\Bs) \] defined by
\[
(\Edot,d^E)\longmapsto(\Edot
          \tens_{\!\A}\Xdot,d^E\# d^X)
\]
where $d^E\# d^X(e\otimes x):=d^E(e)\cdot x+(-1)^{|e|}e\otimes d^X(x),$
and here $\cdot$ denotes the action of $\Adot$ on $\Xdot\tensB\Bdot$.

On homs, the functor is just $\phi\mapsto\phi\otimes_\Bc 1_X$, and one checks directly that this is a morphism of complexes, so that $X_*$ is indeed a dg-functor.

Now for the equivalences.
\begin{defn}\label{D:hMoritaBimodule} An \textbf{h-Morita equivalence twisted bimodule} is a twisted $(\As$-$\Bs)$-bimodule $(X^\bullet,d^X)$ such that there exists a dual $(\Bs$-$\As)$-twisted bimodule $(Y^\bullet,d^Y)$ and \textbf{homotopy equivalences}
\[ X^\bullet\tensB Y^\bullet\approx\As\text{ and }Y^\bullet\tensA X^\bullet\approx\Bs. \]
By homotopy equivalence, we mean that for every $(E,d^E)\in\Mod(\As)$, the objects $Y_*\circ X_*(E^\bullet,d^E)$ and $\A_*(E^\bullet,d^E)$ are homotopy equivalent in $\Mod(\As)$ in a functorial way.  When there exists such a bimodule we say $\As$ and $\Bs$ are \textbf{h-Morita equivalent}.

We call $(X^\bullet,d^X)$ a \textbf{Morita equivalence twisted bimodule} (without the h) when homotopy equivalence is replaced by isomorphism, and in this case the two curved dga's are called \textbf{Morita equivalent}.
\end{defn}
In practice, it will be seen that nontrivial dualities (such as those between noncommutative tori and gerbes on dual-tori) are implemented by h-Morita bimodules, while the more restrictive Morita bimodules implement ``changes of presentation'' (see Section \ref{S:threelevels}) in a sense analogous to a change of groupoid presentation of a stack.

Now we will state and prove our main theorem.
\begin{thm}Let $\As$ and $\Bs$ be curved dga's and $X$ an h-Morita equivalence twisted bimodule.  Then $X_*$ takes $\hPa$ into $\hPb$ and is a dgb-quasi-equivalence.
\end{thm}
\begin{proof} Let $M\in\hPa$ be given.  First note that $X$ is an essential $\Bc$-module so $X_*(M)$ is as well. Now, by definition, the underlying complex $U(M)$ of $M$ is an h-nuclear complex of $\Ac$-modules.  Thus the proof of Theorem \ref{T:dgMorita} applied to $U(M)$ and $U(X_*(M))=U(X)_*(U(M))$ shows that the underlying complex of $X_*(M)$ is h-nuclear.

We just showed that the image of $X_*$ does indeed consist of h-nuclear $\Bs$-modules.  It follows immediately from the definition of h-Morita equivalence that $X_*:\hPa\to\hPb$ is a quasi-equivalence of dgb-categories.
\end{proof}

\section{ Bornological properties of twisted Dolbeault and de Rham algebras }\label{S:bornoprops}

Now that the definition and first properties of the h-perfect category of a nonunital curved dga have been defined, we return attention to gerbes with connection.  In this section we will show that they fit into the framework of h-perfect categories.  More precisely, in this section we construct bornologies on algebras such as the convolution algebra $C_c^\infty(\Gc)$ of a groupoid and the twisted de Rham and Dolbeault algebras of Section \ref{S:twistedDR}.  We show that the resulting bornological algebras are complete, quasi-unital, and multiplicatively convex.  Lastly, we show that groupoid Morita equivalences induce Morita equivalences of smooth convolution algebras.
\newline\newline
\noindent\underline{The bornology on $C_c^\infty(\Gc;\sigma)$.}\indent  Let $\Gc$ be a Lie groupoid with smooth 2-cocycle $\sigma:\Gc_2\to U(1)$.  As a vector space, $C_c^\infty(\Gc;\sigma)$ is just the smooth compactly supported functions on the manifold $\Gc_1$, so we endow it with the LF-space structure given by Equation \eqref{E:seminorms}.  For the bornology on $C_c^\infty(\Gc;\sigma)$ we take the precompact bornology associated to the LF-space.

There are Lie groupoids for which this bornology is not multiplicatively convex (see Remark \ref{R:nonconvex}).  However, in the case that $\Gc$ is proper and \'etale we will be able to choose the seminorms for the LF-structure for $C_c^\infty(\Gc;\sigma)$ in a special way so that multiplicative convexity can be shown to hold.  Let us describe the method here.

Assume $\Gc$ is proper and \'etale.  If $K$ is a compact subset of $\Gc_0$, then properness ensures that the subgroupoid $\Gc_K^K:=(s\times t)^{-1}(K\times K)$ is itself compact.  Then the set $C_K^\infty(\Gc;\sigma)$ of functions in $C_c^\infty(\Gc;\sigma)$ with support in $\Gc_K^K$ actually forms a subalgebra.  We endow $C_K^\infty(\Gc;\sigma)$ with a Fr\'echet structure in the following way.  First, choose a finite collection of vector fields on the manifold $\Gc_1$ which span the tangent bundle of a neighborhood of $\Gc_K^K$ in $\Gc$.  If $X$ is a vector field in this collection, then add to the collection the vector field $\iota_*X$, whose flow is the inverse (groupoid inverse) of the flow associated to $X$.  Then use this resulting collection $\{Y_i\}$ to define the seminorms
\begin{equation}\label{E:seminorms} \rho^K_n(a):=\sum_{|I|\leq n}\sup_{g\in \Gc_K^K}|Y_Ia(g)|\ \ \text{ for } a\in C_c^\infty(\Gc;\sigma) \end{equation}
where $I=(i_1\dots i_k)$ denotes a multi-index of length $|I|=k$ and $Y_I:=Y_{i_1}\circ\cdots\circ Y_{i_k}$.  The point of adding the ``inverse'' vector fields is that each $\rho^K_n$ is invariant under the algebra involution $a\mapsto a^*(g):=\overline{\sigma(g,g^{-1})a(g^{-1})}$.  This is a Fr\'echet space and will be shown in the proof of Theorem \ref{T:bornoProps} to be an admissible Fr\'echet algebra.  Then exhausting $\Gc_0$ by compacts $K_n$ induces an exhaustion of $\Gc$ by compact subgroupoids $\Gc_{K_n}^{K_n}$ and allows us to write $C_c^\infty(\Gc;\sigma)$ as a limit of admissible Fr\'echet algebras.  This implies, according to Lemma \ref{L:LFborno2}, that $C_c^\infty(\Gc;\sigma)$ with its precompact bornology is multiplicatively convex.

Of course the LF structure on $C_c^\infty(\Gc;\sigma)$ is independent of the choice of exhaustion of $\Gc_1$ by compact subsets; the special choices were just a tool for the proof.
\newline\newline
\noindent\underline{The bornology on twisted de Rham and Dolbeault algebras}  Let $\Gc$ be a proper \'etale groupoid and let $\As(\Gc,(\sigma,\theta,B))=(\Adot,d,c)$ be a twisted de Rham dga.  Recall that $\Adot:=\Gamma^\infty_c(\Gc,\sigma;t^*(\wedge^\bullet T^*_{\Gc_0,\Cb}))$ so $\Ac:=\Ac^0=C_c^\infty(\Gc;\sigma)$.  We define the bornology on $\Adot$ as follows.  Choose a smooth family $||\cdot ||$ of Banach algebra norms on $\wedge^\bullet T_\Gc^*$ (such norms exist since this is a bundle of finite dimensional algebras), then say $S\subset\Adot$ is bounded if $||S||\subset C_c^\infty(\Gc)$ is bounded.  The resulting bornology does not depend on the norm.  This follows easily from two facts.  The first is that any two norms, when restricted to a compact subset $K\subset\Gc$, are uniformly equivalent (in the sense that there are positive constants $C$ and $D$ such that $C||\cdot ||_1\leq ||\cdot ||_2\leq D||\cdot ||_1$ over K).  The second is that a bounded subset of $\Adot$ (with respect to either norm) must consist of functions whose combined support lies in some compact set (since that is the case for $C_c^\infty(\Gc;\sigma)$).  Note that this bornology is precisely the precompact bornology associated to the obvious LF-space structure on $\Adot$.  The bornology for a twisted Dolbeault algebra is defined in exactly the same way.

\begin{thm}\label{T:bornoProps} Let $\As(\Gc,(\sigma,\theta,B))=(\Adot,d,c)$ be a twisted Dolbeault or de Rham dga associated to a gerbe with connection, as in Definition \ref{D:dolbeaultAlg} or \ref{D:twisteddeRhamAlg}. Suppose that $\Gc$ is proper.  Then $\Ac$ and $\Adot$ have natural bornological algebra structures for which the following properties hold:
\begin{enumerate}
\item  The bornology only depends on $\Gc$, not on $(\sigma,\theta,B)$.
\item $\Ac$ is a bornological subalgebra of $\Adot$ and $\Adot$ is an essential $\Ac$-module.
\item $\Ac$ and $\Adot$ are complete, convex, quasi-unital, multiplicatively convex bornological *-algebras.
\end{enumerate}
\end{thm}
\begin{proof}  We prove the twisted de Rham dga case, the proof for the twisted Dolbeault dga case is exactly the same.

The bornologies were described in the preceding paragraphs and by construction they only depend on the differentiable structure of the groupoid, so the first statement is proved.

For the second statement, $\Ac$ is obviously a bornological subalgebra of $\Adot$, while the essentialness follows because for any $a\in\Adot$, we can choose a $\phi\in C_c^\infty(\Gc;\sigma)$ whose support $supp(\phi)$ lies in $\Gc_0$ and such that $\phi=1$ on $t(supp(a))$, and any such function satisfies $\phi*a=a$.

Now for the third statement.  Convexity and completeness as bornological spaces is automatic for $\Ac$ and $\Adot$ because these properties hold for the precompact bornology on any LF-space.  Also, the involution is bounded since the seminorms of Equation \eqref{E:seminorms} are *-invariant.

Suppose now that $\Ac$ is quasi-unital and multiplicatively convex.  Then a splitting of the multiplication map for $\Ac$ induces one for $\Adot$ and an approximate identity for $\Ac$ is also one for $\Adot$, so $\Adot$ is quasi-unital.  Also, if $S$ is bounded in $\Adot$ then $||S||$ is bounded in $\Ac$ so some multiple $\lambda ||S||$ is multiplicatively closed in $\Ac$, so by definition $\lambda S$ is multiplicatively closed in $\Adot$.  Thus $\Adot$ is multiplicatively convex.

It remains to show that $\Ac$ is quasi-unital and multiplicatively convex.
\newline\newline
\noindent\underline{Quasi-unitality:}   Choose a locally finite sequence $\{\phi_n\}$ of smooth compactly supported functions on $\Gc_0$, satisfying $\sum\phi_n^2(e)=1$ for all $e\in\Gc_0$.  Then the left $\Ac$-module map
\[
\Ac\ni f\mapsto \sum fs^*\phi_n\tens\phi_n\in\Ac^+\tens\Ac;\quad fs^*\phi_n(\gamma):=f(\gamma)\phi_n(s\gamma)
\]
is a splitting of the multiplication map, and $f\mapsto\sum\phi_n\tens r^*\phi_n f$ gives the right $\Ac-$module splitting.  Furthermore, the sequence $e_k:=\sum_{n<k}\phi_k^2$, when viewed as a function on $\Gc$ with support in $\Gc_0$, is an approximate identity.  Note that the twisting $\sigma$ is irrelevant because the $\phi_n$ are supported on $\Gc_0$ and $\sigma(\eta,\tau)= 1$ whenever either $\eta$ or $\tau$ is a unit.
\newline\newline
\noindent\underline{Multiplicative convexity:}  First consider the case that $\sigma=1$.  For any compact set the algebra $C_K(\Gc)$ of continuous functions with support in $\Gc_K^K=(s\times t)^{-1}(K\times K)$ with sup-norm is a Banach algebra for the convolution multiplication.  Since the groupoid is \'etale, derivatives of functions can only be taken in the ``horizontal'' direction.  Consequently any vector field on $\Gc_1$ acts by derivations for the convolution multiplication:  $X(a*b)=X(a)b+aX(b)$.  Then Puschnigg's derivation lemma (\cite{Pus} p.118) applies: some countable family $\{X_i\}$ of vector fields has $C_K^\infty(\Gc)$ as its common domain of iterated applications thus $C^\infty_K(\Gc)$ is an admissible Fr\'echet algebra.  Exhausting $\Gc$ by compact subgroupoids, one exhibits $\Ac$ as a limit of admissible Fr\'echet algebras, so Lemma \ref{L:LFborno2} implies that $\Ac$ with its precompact bornology is multiplicatively convex.

To treat the case $\sigma\neq 1$, we want to show that $C^\infty_K(\Gc)$ is an admissible Fr\'echet algebra even with the $\sigma$-twisted multiplication.  Clearly this will finish the proof.

We would like to use the derivation lemma, but vector fields are no longer derivations for the twisted multiplication.  Here the connection saves the day.

Define
\[   X^\nabla(a)_\gamma:=da(X)_\gamma+\theta(X)_\gamma a_\gamma   \]
then $X^\nabla$ is no longer a vector field, but it is a (densely defined) derivation on the Banach algebra $C_K(\Gc;\sigma)$ with sup-norm.  Applying the derivation lemma yields a multiplicatively convex Fr\'echet algebra $B$, which is the common domain of finite iterations of $X_i^\nabla$'s.  If we can show that $B=C_K^\infty(\Gc)$ we are done.  We claim that the families of seminorms induced by $\{X_i\}$ and $\{X_i^\nabla\}$ are mutually continuous with respect to each other and thus do indeed define the same Fr\'echet space.  But one can quickly check that
\[ ||X_1\cdots X_n(a)||\leq \sum_{I\subset\{1,...n\}} c_I||X^\nabla_I(a)|| \]
where the positive coefficients $c_I$ depend on the $X_I$ and $\theta$ and derivatives of those, but not on $a$, and similarly,
\[ ||X^\nabla_1\cdots X^\nabla_n(a)||\leq \sum_{I\subset\{1,...n\}} \tilde{c}_I||X_I(a)||. \]  This finishes the proof.
\end{proof}
From now on \textbf{we always assume $\As(\Gc,(\sigma,\theta,B))$ is equipped with the above bornology.}

\begin{rem}  The multiplicative convexity in Theorem \ref{T:bornoProps} cannot be expected to hold in general for the non-proper case.  For example $C_c^\infty(\Rb\arrows *)$ is not multiplicatively convex.  To see this, choose any $f\in C_c^\infty(\Rb\arrows *)$, then the singleton $\{f\}$ is a precompact set, but its multiplicative closure cannot be precompact because the sequence $\{f^k\}_{k=N,N+1,...}$ does not have support in any compact set.
\end{rem}\label{R:nonconvex}
\begin{cor} Let $(\Gc,(\sigma,\theta,B))$ present a gerbe with connection on a proper \'etale groupoid $\Gc$ or a gerbe with $\dbar$-connection on a proper complex \'etale groupoid.  Then $\As(\Gc,(\sigma,\theta,B))$ is a nonunital bornological really curved dga.
\end{cor}
\begin{proof} Everything has been shown in Theorem \ref{T:bornoProps} except that the derivation is a bornological map and that $B\in M(\Adot)$.  These last facts are obvious.
\end{proof}
Now we want to show how the assignment (groupoid$\rightsquigarrow$groupoid algebra) behaves under Morita equivalence.

\begin{thm}\label{T:Morita} Suppose two proper \'etale groupoids $\Gc$ and $\Hc$ are Morita equivalent via a bimodule $P$.  Then we have $C_c^\infty(\Hc)\simeq C_c^\infty(P)\!\tens\!_{C_c^\infty(\Gc)}C_c^\infty(P)$.  Thus $C_c^\infty(\Hc)$ is Morita equivalent to $C_c^\infty(\Gc)$ via the bimodule $C_c^\infty(P)$.
\end{thm}
\begin{proof} Recall that the bimodule structure on $C_c^\infty(P)$ is induced by the groupoid actions.  For example given $\xi\in C_c^\infty(P)$ and $b\in C_c^\infty(\Hc)$, $\xi\cdot b(p):=\sum\xi(ph)b(h^{-1})$.

The map $C_c^\infty(P)\tens C_c^\infty(P)\to C_c^\infty(\Hc)$ can be written as a pairing
\[ \langle\ ,\ \rangle_{\negmedspace\Hc}:C_c^\infty(P)\tens C_c^\infty(P)\to C_c^\infty(\Hc) \]
\[ \xi\tens \eta\mapsto\langle \xi,\eta\rangle_{\negmedspace\Hc}(h):=\sum_{\pi(p)=rh}\overline{\xi(p)}\eta(ph). \]
and there is also a $C_c^\infty(\Gc)$-valued pairing
\[  _\Gc\negmedspace\langle\ ,\ \rangle:C_c^\infty(P)\tens C_c^\infty(P)\to C_c^\infty(\Gc) \]
with the property that for every $\xi,\eta,\zeta\in C_c^\infty(P)$, the equation $\xi\cdot\langle\eta,\zeta\rangle_{\negmedspace\Hc}=\ _\Gc\!\langle\xi,\eta\rangle\cdot\zeta$ holds.  Of course $\langle\ ,\ \rangle\!_\Hc$ descends to the quotient $C_c^\infty(P)\!\tens\!_{C_c^\infty(\Gc)}C_c^\infty(P)$.  We'll show that the pairing is surjective and has an inverse which is well defined in the quotient.

Choose a locally finite cover $\{U_i\}_{i=1,2,...}$ of $\Hc_0$ for which the base map $P\mto{\pi}\Hc_0$ has sections $s_i:U_i\to P$.  Then choose functions $\phi_i\in C_c^\infty(\Hc_0)$ with support in $U_i$ such that $\sum_{i}\phi_i(x)^2=1$ for $x\in\Hc_0$.

The sequence $e_n:=\sum_{i=1...n}\phi_i^2$ is an approximate identity (in the bornological sense) for $C_c^\infty(\Hc)$.  Indeed, let $S\subset C_c^\infty(\Hc)$ be bounded.  Then $S$ is a family of functions whose combined support lies in a compact set.  Since the cover $\{U_i\}$ is locally finite and $K$ is compact, only finitely many of the $\{U_i\}$ intersect $K$.
Thus the sequence $\{e_n\}$ stabilizes on $S$, that is, there is an $N\in\Nb$ such that $e_nf=f$ for all $f\in S$ and all $n\geq N$.

The same argument shows that if $T\subset C_c^\infty(P)$ is a bounded set then there is an $N\in\Nb$ such that $e_n\xi=\xi$ for all $\xi\in T$ and all $n\geq N$.   We could say $\{e_n\}$ is an approximate identity for the action of $C_c^\infty(\Hc)$ on $C_c^\infty(P)$.

The $\{\phi_i\}$ induce smooth functions on $P$ by the formula
\[\tilde{\phi}_i(p):=\begin{cases}
\phi(\pi x)\quad \text{ if } s_i(\pi x)=p \\
0\ \ \ \ \quad\quad      \text{otherwise.}
\end{cases}\]
One checks that $e_n=\sum_{i=1...n}\langle\tilde{\phi}_i,\tilde{\phi}_i\rangle_\Hc$, so if $a\in C_c^\infty(\Hc)$, we can write $a=\lim_{n\to\infty}e_na$.  (This limit is also equal to $e_ka$ for all $k$ larger than some $N$, so there is no question about convergence.)  But then
\[ a=\sum_{i=1...\infty}\langle\tilde{\phi}_i,\tilde{\phi}_i\rangle_\Hc a
    =\sum_{i=1...\infty}\langle\tilde{\phi}_i,\tilde{\phi}_ia\rangle_\Hc \]
so the pairing is surjective.  But in fact, for $\xi,\eta\in C_c^\infty(P)$ we have
\[ \xi\tens \eta=\lim_{n\to\infty}(e_n\xi)\tens \eta= \sum\langle\tilde{\phi}_i,\tilde{\phi}_i\rangle_\Hc \xi\tens \eta
=\sum\tilde{\phi}_i\ \!_\Gc\!\langle\tilde{\phi}_i,\xi\rangle\tens \eta. \]
But in $C_c^\infty(P)\tens_{ C_c^\infty(\Gc)} C_c^\infty(P)$ this last expression is
\[ =\sum\tilde{\phi}_i\tens\ \!_\Gc\!\langle\tilde{\phi}_i,\xi\rangle \eta
   =\sum\tilde{\phi}_i\tens\tilde{\phi}_i\langle \xi,\eta\rangle_\Hc.  \]
This means that the assignment
\[ C_c^\infty(\Hc)\ni a\to \sum_{i=1...\infty}\tilde{\phi}_i\tens\tilde{\phi}_i a\in C_c^\infty(P)\tens_{ C_c^\infty(\Gc)} C_c^\infty(P)\]
is an inverse to the pairing.

We still need to show that the pairing and its inverse are bounded.  Well, if the pairing is a continuous map for the LF-topologies, then it is bounded for the precompact bornology, and it is indeed continuous because the properness of the $\Gc$-action on $P\times_{\Gc_0}P$ implies that for $F\in C_c^\infty(P\times_{\Gc_0}P)$, the sum induced by the pairing $\langle F \rangle\!_\Hc(h)=\sum_{g}F(gp,gph)$ is always finite, in fact uniformly finite on a neighborhood of $h\in\Hc$.

For the inverse, let $S\subset C_c^\infty(\Hc)$ be bounded.  Then there is an $N$ such that the inverse image of $S$ is $(\sum_{i=1\dots N}\tilde{\phi}_i\tens\tilde{\phi}_i)S$.  This is a bounded set in \linebreak $C_c^\infty(P)\tens_{C_c^\infty(\Gc)}C_c^\infty(P)$ because the right action of $C_c(\Hc)$ is bounded.
\end{proof}
There is also a twisted version of Theorem \ref{T:Morita}.  It is the special case of Theorem \ref{T:invariance}.  In the notation of that theorem, it is the case in which $\omega_\Gc=(\sigma_\Gc,0,0)$ and $\omega_\Hc=(\sigma_\Hc,0,0)$ and $\psi=(\alpha_\Gc,\alpha_\Hc,0)$, and it implies $C_c^\infty(P)$ is a $C_c^\infty(\Gc;\sigma_\Gc)$-$C_c^\infty(\Hc;\sigma_\Hc)$-Morita equivalence bimodule (though not with the same bimodule structure).


\section{Three levels of equivalence}\label{S:threelevels}
In this section we study a hierarchy of equivalences between curved dga's as it applies to twisted Dolbeault and de Rham algebras.  The strongest form of equivalence between two curved dga's is of course isomorphism, the next is the equivalence induced by a Morita equivalence twisted bimodule, and the weakest (yet most interesting) is the quasi-equivalence of associated h-perfect categories which is induced by an h-Morita twisted bimodule.  For twisted Dolbeault and de Rham algebras, which are defined via certain 2-cocycles, we show that cohomologous cocycles determine isomorphic curved dga's.  Next, we describe a way to compare 2-cocycles on Morita equivalent groupoids and show that when the cocycles are ``cohomologous'' with respect to this comparison, there is a natural Morita equivalence twisted bimodule over the curved dga's.   Thus the Morita equivalence class of these dga's only depends on the image of the 2-cocycle in stack cohomology, so the dga's can indeed be regarded as presenting (stack theoretic) gerbes with connection.  The third type of equivalence in this setting can be interpreted as Mukai duality for gerbes with connection.  In Section \ref{S:torus} we take up a specific example of this, showing that any flat gerbe on a torus is dual to a noncommutative complex dual-torus.

\subsection{The isomorphism corresponding to gauge transformation}
\begin{prop} Let $\Gc$ be an \'etale groupoid $($either smooth or complex\thinspace$)$.  Then two twisted de Rham or Dolbeault dga's, $\As(\Gc,(\sigma',\theta',B'))=(\Ac_1^\bullet,d_1,c_1)$ and $\As(\Gc,(\sigma,\theta,B))=(\Ac_2^\bullet,d_2,c_2)$, are isomorphic as curved dga's $($see Definition \ref{D:curveddgastuff}\thinspace$)$ whenever the 2-cocycles $(\sigma',\theta',B')$ and $(\sigma,\theta,B)$ are cohomologous.
\end{prop}
\begin{proof}  We will prove the smooth version, the proof for the complex version is identical.  Choose a \v{C}ech-Deligne 1-cocycle $(\alpha,\beta)\in C^1(\Gc;\Deli)$ such that \[(\sigma',\theta',B')=(\sigma,\theta,B)+D(\alpha,\beta)
   \equiv(\delta\alpha\sigma,\dlog\alpha-\delta\beta+\theta,\dr\beta+B)\]
We ll show that the map
\[ M_\alpha :\Ac^\bullet_1\to \Ac^\bullet_2,\quad M_\alpha(f)(\gamma):=\alpha(\gamma)f(\gamma) \]
is an isomorphism of graded algebras and that the pair $(M_\alpha ,\beta)$ is a morphism of curved dga's.

First, multiplication by $\alpha$ is an isomorphism of bornological vector spaces since it takes values in $U(1)$.  Next we verify that $M_\alpha$ intertwines the multiplications:
\begin{align*}
M_\alpha (f*_1g)(\gamma)
&=\alpha(\gamma)\sum_{\eta\tau=\gamma}\delta\alpha\sigma(\eta,\tau)f(\eta)\wedge g(\tau) \\
&=\sum_{\eta\tau=\gamma}\alpha(\eta\tau)\delta\alpha(\eta,\tau)\sigma(\eta,\tau)f(\eta)\wedge g(\tau)\\
&=\sum_{\eta\tau=\gamma}\alpha(\eta)\alpha(\tau)\sigma(\eta,\tau)f(\eta)\wedge g(\tau)\\
&=\sum_{\eta\tau=\gamma}\sigma(\eta,\tau)\alpha(\eta)f(\eta)\wedge\alpha(\tau)g(\tau)\\
&=(M_\alpha(f)*_2M_\alpha(g))(\gamma)
\end{align*}
Where we use in the third line that $\alpha(\eta\tau)\delta\alpha(\eta,\tau)=\alpha(\eta)\alpha(\tau)$. \\

Now we check compatibility with differentials:
\begin{align*}
M_\alpha(d_1f)(\gamma)
=\alpha(\gamma)\{\dr f(\gamma) + \dlog(\alpha)(\gamma)\wedge
f(\gamma)+\theta(\gamma)f(\gamma)-\delta\beta(\gamma)\wedge
f(\gamma)\}
\end{align*}
Expanding and using $\dlog\alpha=\alpha^{-1}\dr\alpha$ and the Leibnitz
rule for $\dr$, this becomes
\begin{align*}
M_\alpha(d_1f)(\gamma)
&=\dr(\alpha f)(\gamma)+\theta(\gamma)\wedge\alpha f(\gamma)-\delta\beta(\gamma)\wedge f(\gamma) \\
&=d_2(\alpha f)(\gamma)-\delta\beta(\gamma)\wedge \alpha f(\gamma) \\
&=d_2(\alpha f)(\gamma)-(\beta(s\gamma)-(\beta(r\gamma))\wedge \alpha f(\gamma)
\end{align*}
and remembering that $\beta$ vanishes off of units so
$(\beta(r\gamma))\wedge f(\gamma)=(\beta *f)(\gamma)$, gives
\begin{align*}
&=d_2(\alpha f)(\gamma)+\beta *\alpha f(\gamma)-(-1)^f\alpha f*\beta \\
&=d_2(M_\alpha f)(\gamma)+[\beta ,M_\alpha f](\gamma)
\end{align*}
as desired.\\
The facts that $\alpha\equiv 1$ on units, $\theta$ vanishes on units, and $\beta$ vanishes \textit{off} of units yield $\theta\wedge\beta=0$, $\beta\wedge\beta=0$ and finally
\[
M_\alpha(h+\dr\beta)(\gamma)=h(\gamma)+\dr\beta(\gamma)=h(\gamma)+\dr\beta(\gamma)+\theta\wedge\beta+\beta^2
\]
which is the second condition for a morphism of curved dga's.
\end{proof}

\subsection{The Morita equivalence from change of gerbe presentation}\label{S:change of presentation}
\ \newline

Now we want a way to compare twisted Dolbeault or de Rham dga's presented over different but Morita equivalent groupoids.  The need for some sort of comparison is clear because Morita equivalent groupoids essentially determine the same stack, and one should be able to say when two gerbes with connection over the same stack are equivalent.  For notational convenience the constructions in this section are made for \'etale groupoids and the associated twisted de Rham dga's, but they are also valid for complex \'etale groupoids and the associated twisted Dolbeault dga's.

The comparison works as follows.  A groupoid Morita equivalence bimodule relating two groupoids $\Gc$ and $\Hc$ gives rise to a canonical bi-simplicial manifold which maps to the simplicial manifolds $B\Gc_\bullet$ and $B\Hc_\bullet$.  This induces maps from the groupoid cohomology of $\Gc$ and $\Hc$ to some cohomology groups associated to the bi-simplicial manifold, and thus gives a way to compare the cohomology of $\Gc$ with that of $\Hc$.  Now twisted de Rham and Dolbeault dga's are given by certain 2-cocycles, and we will show that when a pair of such 2-cocycles on Morita equivalent groupoids have cohomologous images under these maps, the associated twisted de Rham or Deligne dga's are Morita equivalent as curved dga's.

To get started, let us fix some terminology regarding Morita equivalence of groupoids.  We will mostly use the notation developed in \cite{C}, which has more detailed descriptions of Morita equivalence as it relates to cohomology.  Let $\Gc\arrows\Gc_0$ and $\Hc\arrows\Hc_0$ be \'etale groupoids.  A \textbf{left $\Gc$-space} is a space over $\Gc_0$, $P\stackrel{\eps}{\to}\Gc_0$, with an action $\Gc\times_{\Gc_0}P\to P$.  A \textbf{left $\Gc$-bundle} over a \textbf{base} $B$ is a $\Gc$-space $P$ together with a surjective submersion $\pi:P\to B$ which is $\Gc$-invariant in the sense that $\pi(gp)=\pi(p)$.  The $\Gc$-bundle $P$ is called \textbf{principal} if the map
\[ \Gc\times_sP\lra P\times_\pi P;\quad (g,p)\mapsto(gp,p)  \]
is an isomorphism.  Right $\Hc$-spaces and $\Hc$-bundles and principality are analogous.  A \textbf{$\Gc-\Hc$ Morita equivalence bimodule} (sometimes called a principal bi-bundle) is a space $P$ which is simultaneously a principal left $\Gc$ bundle over the base $X=\Hc^0$ and a principal right $\Hc$-bundle over the base $Y=\Gc_0$.

Write $B\Gc_\bullet$ for the simplicial nerve of $\Gc$.  This is the simplicial manifold whose $k^{th}$ piece is $\Gc_k$.  If $P$ is a Morita equivalence bimodule, then it determines a bi-simplicial manifold
\[BP_{k,l}=\Gc_k\times_{\Gc_0}P\times_{\Hc_0}\Hc_l.\]
The two collections of boundary maps on $BP_{\bullet\bullet}$ come from the boundary maps for $B\Gc_\bullet$ and $B\Hc_\bullet$.  The two obvious augmentations
\[ BP_{k,0}\to B\Gc_k \text{ and } BP_{0,l}\to B\Hc_l  \]
induce morphisms of complexes
\begin{equation}\label{coaugs}
C(\Hc;\Deli)\to C(P;\Deli)\leftarrow C(\Gc;\Deli)
\end{equation}
Here $\Deli$ denotes the Deligne complex (Definition \ref{D:DeligneComplex}):
\[ (\Deli^0\tof{\ddeli}\Deli^1\tof{\ddeli}\Deli^2)=
   (\Uone\tof{\dlog}2\pi\sqrt{-1}\Ac^1\tof{\dr}2\pi\sqrt{-1}\Ac^2),  \]
$C(\Gc;\Deli)$ and $C(\Hc;\Deli)$ are the \v{C}ech-Deligne bi-complexes defined in Section \ref{S:CechDelingeCoho}, that is they are the complexes induced by the differential on $\Deli$ and the simplicial boundary maps of $B\Gc_\bullet$ and $B\Hc_\bullet$.  Then $C(P;\Deli)$ is the tricomplex one gets from the differential on $\Deli$ and the two simplicial boundary maps on $BP_{\bullet\bullet}$.  The three differentials on $C(P;\Deli)$ will be used, so let us write them out explicitly:
\[ C^{klm}=C^{klm}(P;\Deli):=\Gamma(\Gc_k\times_{\Gc_0}P\times_{\Hc_0}\Hc_l;\Deli^m) \]
The three differentials are: $\begin{cases}
\ddeli:C^{klm}\to C^{k,l,m+1} \\
\delta_\Gc\ :C^{klm}\to C^{k+1,l,m}\\
\delta_\Hc\ :C^{klm}\to C^{k,l+1,m} \\
\end{cases}$\newline
where for $f\in C^{klm}$,
\begin{align*}
\quad\delta_\Gc f(g_1,\dots,g_{k+1},p,h's)&:= f(g_2,\dots,g_{k+1},p,h's) \\
& \!\!\!\!\!\!\!+\sum_{i=1\dots k}(-1)^if(\dots,g_ig_{i+1},\dots,p,h's) +(-1)^{k+1}f(g_1,\dots,g_k,g_{k+1}p,h's)\\
\quad\delta_\Hc f(g's,p,h_1,\dots,h_{l+1})&:= f(g's,ph_1,h_2,\dots,h_{l+1}) \\
& \!\!\!\!\!\!\!+\sum_{i=1\dots l}(-1)^if(g's,p,\dots ,h_ih_{i+1},\dots)+(-1)^{l+1}f(g's,p,h_1,\dots,h_l).
\end{align*}
The total differential on $\tot(C(P;\Deli))$ is given by
\[ \DP:=\sum\DP|_{C^{klm}} \text{ where } \DP|_{C^{klm}}:=
   \ddeli +(-1)^m((-1)^l\delta_\Gc+\delta_\Hc).\]
One should remember though, that the minus signs in the formula for $\DP$ means inverse for the $U(1)$-component of the Deligne complex.

Let $\As(\Gc;\omega_{\Gc})$ and $\As(\Hc;\omega_\Hc)$ be twisted de Rham dga's corresponding to \v{C}ech-Deligne 2-cocycles $\omega_\Gc$ and $\omega_\Hc$.  Now suppose the images $\widetilde{\omega_\Gc}$, and $\widetilde{\omega_\Hc}$ under the coaugmentations \eqref{coaugs} are cohomologous, that is, there is some $\psi\in\tot C(P;\Deli)^1$ satisfying $\DP\psi=\widetilde{\omega_\Gc }-\widetilde{\omega_\Hc }$.  Then from this data we can define a twisted $\As(\Gc;\omega_{\Gc})$-$\As(\Hc;\omega_\Hc)$-bimodule which induces a Morita equivalence of curved dga's.
\begin{thm}\label{T:invariance}
Suppose $(\Gc,\omega_\Gc)$ and $(\Hc,\omega_\Hc)$ are a pair of proper \'etale groupoids with Deligne 2-cocycle and we are given a pair of data $(P,\psi)$, where $P$ is a $\Gc-\Hc$-Morita equivalence bimodule and $\psi\in\tot C(P;\Deli)^1$ satisfies $\DP\psi=\widetilde{\omega_\Hc}-\widetilde{\omega_\Gc}$.
Then there is a twisted $\As(\Gc;\omega_{\Gc})$-$\As(\Hc;\omega_\Hc)$-bimodule associated to $(P,\psi)$ which induces a Morita equivalence of curved dga's.
\end{thm}
\begin{cor}
The dgb-categories $\hPa(\Gc,\omega_\Gc)$ and $\hPa(\Hc,\omega_\Hc)$ are quasi-equivalent.
\begin{flushright} $\square$ \end{flushright}
\end{cor}
\begin{proof}(of Theorem)  By definition, the cochain $\psi$ takes the form of a triple which we write
\[  \psi=(\omega,\alpha_\Gc,(\alpha_\Hc)^{-1})\in C^{010}\times C^{001}\times C^{100} \]
The twisted de Rham dga's at hand are
\[ \As(\Gc;\omega_\Gc)=(\Ac_1^\bullet,d_1,c_1)\text{ and }\As(\Hc;\omega_\Hc)=(\Ac_2^\bullet,d_2,c_2) \]
The $\As(\Gc;\omega_\Gc)$-$\As(\Hc;\omega_\Hc)$-twisted bimodule is $X=X(P,\psi)=(C_c^\infty(P),\nabla)$ (here $C_c^\infty(P)$ is a complex supported in degree zero) with the actions and $\nabla$ defined as follows:\newline
For $p\in P, \xi\in C_c^\infty(P), h\in \Hc, g\in \Gc, a's\in\Ac^\bullet_1, f\in\Ac_2,$ and $b's\in\Ac^\bullet_2$, set
\begin{align}
    &\nabla\xi(p)=\dr\xi(p)+\omega(p)\xi(p)\in C_c^\infty(P)\otimes_{\Ac_2}\Ac^\bullet_2 \\
    &\xi\cdot f(p)=\sum\xi(ph^{-1})\alpha_\Hc(ph^{-1},h)f(h)\in C_c^\infty(P) \label{E:1}\\
    &a\cdot\xi(p)=\sum a(g)\alpha_\Gc(g,g^{-1}p)\xi(g^{-1}p)\in C_c^\infty(P)\otimes_{\Ac_2}\Ac^\bullet_2 \label{E:2}
\end{align}
In these formulas we are tacitly identifying differential forms on $P$ with elements of $C_c^\infty(P)\otimes_{\Ac_2}\Ac^\bullet_2.$
Now we want to compare $\omega_\Gc$ with $\omega_\Hc$ in $C(P;\Deli)$.
So write $\omega_\Gc=(\sigma_\Gc,\theta_\Gc,B_\Gc)\in C^2(\Gc;\Deli)$ and $\omega_\Hc=(\sigma_\Hc,\theta_\Gc,B_\Hc)\in C^2(\Hc;\Deli)$.  Then by definition,
\[
\widetilde{\omega_\Gc}-\widetilde{\omega_\Hc}=(\ \widetilde{\sigma_\Gc},\ \widetilde{\sigma_\Hc}^{-1}\!,\ \widetilde{\theta_\Gc},-\ \widetilde{\theta_\Hc},\ \widetilde{B_\Gc}-\widetilde{B_\Hc},1)\]
\[\hspace{1.5in}\in\! D^{200}\!\!\times\! D^{020}\!\!\times\! D^{101}\!\!\times\! D^{011}\!\!\times\! D^{002}\!\!\times D^{110}\]
while by definition of the differential on $\tot(C(P;\Deli)$,
\[ \DP\psi=(\ \delta_\Gc\alpha_\Gc,\ \delta_\Hc\alpha_\Hc^{-1},
   \ \dlog\alpha_\Gc-\delta_\Gc\omega,\ \dlog(\alpha_\Hc^{-1})-\delta_\Hc\omega,\ \dr\omega,\ \delta_\Hc\alpha_\Gc(\delta_\Gc\alpha_\Hc^{-1})^{-1}) \]

Then the equation $\DP\psi=\widetilde{\omega_\Gc}-\widetilde{\omega_\Hc}$ breaks up into six equations that conspire to make the properties of a twisted bimodule hold for $X(P,\psi)$.  In the left column we put the equation, and in the right column we put the resulting bimodule property (from Definition \ref{bimod}) that is enforced.
\begin{align*}
&\delta_\Hc\alpha_\Hc=\widetilde{\sigma_\Hc}
&\iff& (\xi\cdot b_1)\cdot b_2=\xi\cdot(b_1b_2) \\
&\delta_\Gc\alpha_\Gc=\widetilde{\sigma_\Gc}
&\iff& (a_1a_2)\cdot\xi=a_1\cdot(a_2\cdot\xi) \\
&\dr\alpha_\Hc=\alpha_\Hc(\widetilde{\theta^\Hc}-\delta_\Hc\omega)
&\iff& \nabla(\xi\otimes b)=\nabla\xi*\cdot b+\xi\otimes d_2b          \\
&\dr\alpha_\Gc=\alpha_\Gc(\widetilde{\theta_\Gc}+\delta_\Gc\omega)
&\iff& \nabla(a\cdot\xi)=(d_1a)\cdot\xi+(-1)^aa\cdot\nabla\xi        \\
&\delta_\Hc\alpha^\Gc\delta_\Gc\alpha^\Hc=1
&\iff& a\cdot(\xi\cdot b)=(a\cdot\xi)\cdot b\\
&\dr\omega=\widetilde{B^\Gc}-\widetilde{B^\Hc}
&\iff& \nabla\circ\nabla(\xi\otimes a)=B_\Gc\cdot(\xi\otimes a)-(\xi\otimes a)\cdot B_\Hc
\end{align*}
 To see that this twisted bimodule is a Morita equivalence, we just need to show that $C_c^\infty(P)$ is a  $C_c^\infty(\Gc;\sigma_\Gc)$-$C_c^\infty(\Hc;\sigma_\Hc)$-Morita equivalence bimodule.  The bimodule structure is given by Equation \eqref{E:1} and Equation \eqref{E:2} restricted to $a\in\Ac_2^0=C_c^\infty(\Hc;\sigma_\Hc)$.  The isomorphism
\begin{equation}\label{E:3} C_c^\infty(P)\tens_{C_c^\infty(\Gc;\sigma_\Gc)}C_c^\infty(P)\To C_c^\infty(\Hc) \end{equation}
is induced from the pairing
\[ C_c^\infty(P)\times C_c^\infty(P)\To C_c^\infty(\Hc);\qquad
   \langle\xi,\eta\rangle_{\!\Hc}(h):=\sum_{p}\overline{\xi(p)\alpha_\Hc(p,h)}\eta(ph). \]
There is an analogous pairing $ _{\Gc\!}\langle\xi,\eta\rangle(g):=
\sum_{p}\xi(gp)\overline{\alpha_\Gc(g,p)\eta(p)}$.  Now the proof that \ref{E:3} is an isomorphism is identical to the proof of Theorem \ref{T:Morita} (which was the case $\sigma_\Gc=\sigma_\Hc=1$) if we substitute in the modified bimodule structures and modified pairings.  In particular, the approximate identities that are used in the proof work without modification since they are supported on units and our modified structures reduce to the old ones on units.
\end{proof}

\subsection{Duality}\label{S:duality}
\ \newline

Let $\As$ and $\Bs$ be bornological curved dga's.  Then the existence of an h-Morita equivalence twisted bimodule (which will induce a quasi-equivalence of h-perfect categories $\hPa\simeq\hPb$) can be taken as a \emph{definition} of duality between $\As$ and $\Bs$.

Note that this type of equivalence is strictly weaker than the equivalence from change of presentation of a gerbe with connection, which is implemented by a Morita equivalence twisted bimodule (without the ``h'').
\subsection{A remark }\label{S:stacks}
\ \newline

The content of Theorem \ref{T:invariance} is that a gerbe with connection is a concept native to \'etale stacks, that is, to Morita equivalence classes of \'etale groupoids.  Indeed, the bimodule $P$ is an equivalence of groupoids and we saw that when a pair of 2-cocycles are cohomologous in the complex associated to $P$ they are different presentations of the same gerbe with connection.

If the manifolds $\Gc_k$ and $\Hc_k\ k=0,1,2...$ are acyclic for each sheaf that appears in the 2-truncated Deligne complex (that is the sheaves $U(1),$ $\A^1$, and $\A^2$), then the inclusions of complexes \ref{coaugs} are quasi-isomorphisms and the groupoid cohomology of either groupoid is the stack cohomology as in \cite{Beh}, and the 2-cocycles are representatives of the same cohomology class on the stack $B\Gc\simeq B\Hc$.  In any case, \v{C}ech-groupoid cohomology maps to stack cohomology, so at worst there may be pairs of groupoid presentations of gerbes with connection which do not appear to be Morita equivalent but for which the associated (stack theoretic) gerbes with connection are equivalent.

So for us good groupoids are those for which the manifolds $\Gc_k,\ k=0,1.2...$ are acyclic for the components of the Deligne complex, and it is for such groupoids that a gerbe with connection is really determined by a 2-cocycle in stack cohomology (with coefficients in the Deligne complex).  The question of whether every \'etale groupoid admits a good refinement is still open to us, but every \v{C}ech groupoid does since every such groupoid can be refined to a \v{C}ech groupoid on a good cover, and for \v{C}ech groupoids on good covers $\Gc_k$ is a contractible manifold for each $k$, and contractible manifolds are acyclic for any sheaf of abelian groups (see, for example, \cite{KS}).

\section{T-duality between gerbes and noncommutative tori}\label{S:torus}

In this section we will prove the T-duality between a flat gerbe on a holomorphic torus and a holomorphic noncommutative dual torus.  The first step will be to prove that every gerbe with flat connection on a manifold can be presented by a commutative curved dga.  After that is done we will recall the construction from \cite{Bl2} of the curved dga associated to a holomorphic noncommutative torus, and apply a previous result to quickly finish the T-duality theorem.

\begin{thm}\label{T:flatGerbes} Let $\Gc\arrows\Gc_0$ be a proper \'etale groupoid which is Morita equivalent to a manifold $X\arrows X$.  Suppose $[(\sigma,\theta,B)]\in H^2(\Gc;\Deli)$ satisfies $\dr B=0$.  Then:
\begin{enumerate}
\item The twisted de Rham dga $\As(\Gc,(\sigma,\theta,B))$ is Morita equivalent to \newline $\As(X\arrows X,(1,0,B'))$ for some $B'\in{(\Ac_X^2)}_{closed}$.
\item The twisted de Rham dga $\As(\Gc,(\sigma,\theta,B))$ is Morita equivalent to \newline
$\As(\Hc,(\sigma',0,0))$ for some \v{C}ech groupoid $\Hc$ of $X$ and locally constant 2-cocycle $\sigma'$.
\item The analogous statements for complex groupoids and twisted Dolbeault dga's are also true.
\end{enumerate}
\end{thm}
\begin{proof}
This proceeds in three steps.  First, let $\{U_i\}$ be a locally finite good cover of $\Gc_0$ which is also a locally finite good cover of $X$ (this makes sense because the quotient map $\Gc_0\to\Gc_0/\Gc_1\simeq X$ is \'etale).  Let $\Hc$ be the refinement of $\Gc$ induced by this cover, that is
\[ \Hc:=\ \coprod_{\langle i,j\rangle}(s^{-1}U_i\cap t^{-1}U_j)\arrows\coprod_i U_i \]
where $s$ and $t$ are the source and target of $\Gc$.  Gluing together the cover induces an obvious morphism $\phi:\Hc\to\Gc$, which is in fact an essential Morita equivalence.  Thus $\phi$ induces a map in cohomology and it is easy to verify (using Theorem \ref{T:invariance}) that $\As(\Gc,(\sigma,\theta,B)$ and $\As(\Hc,\phi^*(\sigma,\theta,B))$ are Morita equivalent curved dga's.

Since $\Hc_0$ is a cover of $X$, $\Hc$ is a \v{C}ech groupoid and its groupoid cohomology is the same as \v{C}ech cohomology of the cover of $X$.  Since the cover is good (i.e. all intersections are contractible) and $\dr B=0$, the usual tic-tac-toe argument can be used to show that $\phi^*(\sigma,\theta,B)\in Z^2(\Hc;\Deli)$ is cohomologous to a cocycle of the form $(\sigma',0,0)$.  Thus $\As(\Hc,\phi^*(\sigma,\theta,B))$ and $\As(\Hc,(\sigma',0,0))$ are isomorphic curved dga's and the second statement is proved.

The tic-tac-toe argument also can be used to show that $\phi^*(\sigma,\theta,B)$ is cohomologous to a cocycle and of the form $(1,0,B'')$, so that $\As(\Hc,\phi^*(\sigma,\theta,B))$ and $\As(\Hc,(1,0,B''))$ are isomorphic curved dga's.

Since the 2-form $B''$ on $\Hc_0$ is $\delta$-closed it is a pullback $q^*B'$ via the projection $q:\Hc_0\To\Hc_0/\Hc_1=X$ of some closed differential 2-form on $X$.  Now let $P$ be a Morita equivalence bimodule between $\Hc$ and $X\arrows X$.  Then necessarily $P=\Hc_0$, and the equation $q^*B'=B''$ implies that the images of $(1,0,B'')\in Z^2(\Hc;\Deli)$ and $(1,0,B')\in Z^2(X\arrows X;\Deli)$ in the double complex $C^{\bullet\bullet}(P;\Deli)$ are cohomologous (see Section \ref{S:change of presentation}).  Thus $\As(\Hc,(1,0,B''))$ and $\As(X\arrows X,(1,0,B'))$ are Morita equivalent as curved dga's.

Then we have a chain of equivalences
\[ \As(\Gc,(\sigma,\theta,B))\sim\As(\Hc,\phi^*(\sigma,\theta,B))
    \sim\As(\Hc,(1,0,B''))\sim\As(X\arrows X,(1,0,B'))\]
which proves the first statement.  The same arguments works in the complex case.
\end{proof}
It is interesting to see what the above theorem says about perfect categories.  Suppose $X$ is a complex manifold and a cohomology class $[(\sigma,\theta,B)]\in H^2(X;\Dol)$ is given.  Here $\Dol$ denotes the 2-truncated  Dolbeault complex.  We can present $(\sigma,\theta,B)$ on a good cover of $X$, with associated \v{C}ech groupoid $\Gc$.  Then if $\dbar B=0$ there exist $\sigma'$ and $B'$ such that $\As(\Gc; (\sigma',0,0)$ and $\As(X\arrows X;(1,0,B'))$ are Morita equivalent.  This implies that the associated perfect categories are equivalent.  The first category is made of complexes of twisted sheaves on $\Gc_0$ with $\dbar$-flat connections.  In other words they are complexes of holomorphic twisted sheaves.  The second category is made of complexes of $C_c^\infty(X)$-modules with $\dbar$-connection whose curving, $B'$, is a holomorphic $(0,2)$-form on $X$.

Still no Mukai duality has appeared.  Let us proceed to that now.  Let $X=V/\Lambda$ be a complex torus.  So $V$ is a complex $n$-dimensional vector space and $\Lambda$ is a lattice subgroup isomorphic to $\Zb^{2n}$.  We will show that if $[(\sigma,\theta,B)]\in H^2(X;\Dol)$ satisfies $\dbar B=0$, then
$\As(X\arrows X;(\sigma,\theta,B))$ is Mukai dual (in the sense of Section \ref{S:duality}) to a curved dga corresponding to a noncommutative holomorphic dual torus.  We will denote this curved dga $\As(\Lambda\arrows *;\sigma')$, where $\sigma'$ is a $U(1)$-valued 2-cocycle on the groupoid $\Lambda\arrows *$.

We recall the definition from (Section 3 \cite{Bl2}) of this unital curved dga.  Write
\[ \As(\Lambda; \sigma):=(\Adot,\dbar,0) \]
where $\Adot:=\Sc(\Lambda;\sigma,\wedge^\bullet V_{1,0})$ is the $\sigma$-twisted convolution algebra of Schwartz functions on $\Lambda$ with coefficients in the exterior algebra of $V_{1,0}$ (here $V_{1,0}$ denotes the $+i$-eigenspace of the complex structure operator on $V\tens_{\!\Rb}\Cb$).  Thus multiplication in $\Adot$ is given by
\[ f*g(\lambda):=\sum_{\lambda_1+\lambda_2=\lambda}f(\lambda_1)\wedge g(\lambda_2)\sigma(\lambda_1,\lambda_2). \]
The derivation $\dbar$ is given by the formula
\[ \dbar f(\lambda):=2\pi i f(\lambda)p_{1,0}(\lambda)\in\Sc(\Lambda;\wedge^1V_{1,0})\ \text{ for } f\in\Sc(\Lambda). \]
Here $p_{1,0}:V\hookrightarrow V\tens_{\!\Rb}\Cb\to V_{1,0}$ is induced by the complex structure on $V$.
The derivation is extended to $\Adot$ according to the Leibnitz rule.

Note that when $\sigma=1$ Fourier transform is an algebra isomorphism between the degree zero component $\Ac$ and the smooth functions on the dual torus $X^\vee:=\Hom(\Lambda; U(1))=\overline{V}^\vee/\Lambda^\perp$.  In fact, for the $\sigma=1$ case this extends to an isomorphism $\Adot\simeq\Gamma^\infty(X^\vee;\wedge^\bullet T_{X^\vee}^{0,1}$) which takes $\dbar$ to the usual  differential on the Dolbeault algebra of $X^\vee$.

We assume $\Sc(\Lambda)$ is equipped with the precompact bornology associated to its usual Fr\'echet structure. Thus $\Sc(\lambda)\subset C^*(\Lambda)$ is the domain of the iterated applications of the Fourier transform of ``$d/dx_i$'' on $C(X^\vee)$, so it follows from Puschnigg's derivation lemma (\cite{Pus}) and Lemma \ref{L:LFborno2} that $\Sc(\Lambda)$ is a multiplicatively convex bornological algebra.  In fact these derivations are derivational for $\sigma$-twisted multiplication, so $\Sc(\Lambda;\sigma)\subset C^*_{red}(\Lambda;\sigma)$ is multiplicatively convex by the same reasoning.  Thus $\Adot=\Sc(\Lambda;\sigma)\tens\wedge^\bullet V_{1,0}$ has a multiplicatively convex bornology as well, extended from the degree zero component by using the fine bornology on $\wedge^\bullet V_{1,0}$.

Now let $B\in (A^{0,2}_X)_{const}\simeq\wedge^2V^{0,1}$ be a constant $(0,2)$-form.  The usual isomorphism $V^{0,1}\to V$ takes $B$ to $\Rs(B):=B+\overline{B}\in\wedge^2V^\vee$.  By restriction to $\Lambda\subset V$, $\Rs(B)$ may be viewed as a group 2-cocycle $\Rs(B):\wedge^2\Lambda\to\Rb$.  Define a $U(1)$-valued cocycle by
\[ \sigma_B:\wedge^2\Lambda\to U(1)\qquad\quad
    \sigma_B(\lambda_1,\lambda_2):=e^{2\pi i\Rs(B)(\lambda_1,\lambda_2)}. \]

\begin{thm}\label{T:BlocksDuality}\textup{(\cite{Bl2} Theorem 3.6)} Let $X=V/\Lambda$ be a complex torus with associated dual torus $X^\vee$.  Then given $B\in(\Gamma(X;\wedge^2T_X^{0,1}))_{const}$ and $\sigma_B\in Z^2(\Lambda;U(1))$ as in the previous paragraph, there is an h-Morita equivalence
\[ \As(X\arrows X;(1,\ 0,\ 2\pi i B))\ \sim\ \As(\Lambda;\sigma_B). \]
\end{thm}
It is worth noting that there are explicit h-Morita equivalence bimodules for this equivalence which are based on deformed versions of smooth sections of a Poincar\'e line bundle, so that in the untwisted case this reduces to (a dg-enhancement of) the usual Poincar\'e sheaf that implements T-duality in complex geometry.

Now combining Theorems \ref{T:flatGerbes} and \ref{T:BlocksDuality} gives:
\begin{thm}$($T-duality\thinspace$)$ Let $\Gc$ be a proper complex \'etale groupoid that is Morita equivalent to a complex torus $X=V/\Lambda$, and let $[(\sigma',\theta',B')]\in H^2(\Gc;\Dol)$ satisfy $\dbar B'=0$.  Then there is an h-Morita equivalence
\[  \As(\Gc,(\sigma',\theta',B'))\sim\As(\Lambda;\sigma_B) \]
where $(1,0,2\pi i B)\in Z^2(X\arrows X;\Dol)$ is cohomologous to $(\sigma',\theta',B')\in Z^2(\Gc;\Dol)$ in the sense of Section \ref{S:change of presentation}.
\end{thm}
The left side is a gerbe with $\dbar$-flat connection on a torus and the right side is a noncommutative holomorphic torus, and in the case $B=0$ this reduces to complex T-duality or Fourier-Mukai duality (the two dualities are related by the classical Fourier transform $\Sc(\Lambda)\simeq C^\infty(X^\vee)$).
Consequently, there is a dg-quasiequivalence:
\[ \psa(\text{\small{flat gerbe on torus}})\sim \psa(\text{\small{noncommutative dual torus}}). \]
This induces an equivalence of triangulated categories
\[ \Ho\psa(\text{\small{flat gerbe on torus}})\simeq \Ho\psa(\text{\small{noncommutative dual torus}}) \]
and reduces in the case $B=0$ to an equivalence of triangulated categories
\[ D_{coh}^b(X)\simeq\Ho\psa(\text{\small{trivial gerbe on }}X)\simeq
    \Ho\psa(\text{\small{commutative }}X^\vee)\simeq D_{coh}^b(X^\vee). \]
There are obvious generalizations of this to families of smooth tori, though we have not yet considered cases with singular torus fibers.

\end{document}